\documentclass[11pt]{article}
\usepackage{caption}
\usepackage{epsf,epsfig,amsfonts,amsgen,amsmath,amstext,amsbsy,amsopn,amsthm
}
\usepackage{ebezier,eepic}
\usepackage{color}
\usepackage{multirow}
\usepackage{mathrsfs}
\usepackage{tikz}
\usepackage{enumerate}
\usepackage{enumitem}
\usepackage{url}
\usepackage{pgf,tikz,pgfplots}
\usepackage{mathrsfs}
\usetikzlibrary{arrows}
\newtheorem{dfn}{Definition}[section]
\newtheorem{thm}[dfn]{Theorem}
\newtheorem{lem}[dfn]{Lemma}

\newtheorem{conjecture}[dfn]{Conjecture}

\newtheorem{Dfn}{Definition}

\newtheorem{prop}[Dfn]{Proposition}

\newcommand{\dP}{{\mathcal{P}}}
\newcommand{\dF}{{\mathcal{F}}}

\newcommand{\dS}{{\mathcal{S}}}
\newcommand{\dT}{{\mathcal{T}}}
\newcommand{\dA}{{\mathcal{A}}}
\newcommand{\dB}{{\mathcal{B}}}
\newcommand{\dC}{{\mathcal{C}}}
\newcommand{\dI}{{\mathcal{I}}}

\newcommand{\dM}{{\mathcal{M}}}
\newcommand{\dN}{{\mathcal{N}}}
\newcommand{\dW}{{\mathcal{W}}}
\newcommand{\bh}{{\backslash}}
\newcommand{\<}{{\prec}}
\newcommand{\s}{{\succ}}
\newcommand{\pq}{{\preceq}}
\newcommand{\sq}{{\succeq}}
\newcommand{\dif}{{\bigtriangleup }}

\def\ls{\mathrm{LS}}
\def\rs{\mathrm{RS}}

\def\ps{\mathrm{ps}}

\newcommand{\1}{{\uppercase\expandafter{\romannumeral1}}}
\newcommand{\2}{{\uppercase\expandafter{\romannumeral2}}}
\newcommand{\3}{{\uppercase\expandafter{\romannumeral3}}}
\newcommand{\4}{{\uppercase\expandafter{\romannumeral4}}}

\let\svthefootnote\thefootnote
\newcommand\blankfootnote[1]{%
	\let\thefootnote\relax\footnotetext{#1}%
	\let\thefootnote\svthefootnote%
}

\begin{document}

\title{Non-repeated cycle lengths and Sidon sequences}
\author{
Jie Ma
~~~~~
Tianchi Yang
}

\date{}

\maketitle

\begin{abstract}
We prove a conjecture of Boros, Caro, F\"uredi and Yuster on the maximum number of edges in a 2-connected graph
without repeated cycle lengths, which is a restricted version of a longstanding problem of Erd\H{o}s.
Our proof together with the matched lower bound construction of Boros, Caro, F\"uredi and Yuster show that
this problem can be conceptually reduced to the seminal problem of finding the maximum Sidon sequences in number theory.
\end{abstract}

\blankfootnote{\noindent School of Mathematical Sciences, University of Science and Technology of China, Hefei, Anhui 230026, China. Research supported in part by NSFC grant 11622110.}

\section{Introduction}
An old problem of Erd\H{o}s since 1975 (see \cite{BM}, p. 247, Problem 11) asks to determine the maximum number $n+f(n)$ of edges
in an $n$-vertex graph in which no two cycles have the same length.
An early result of Shi \cite{shi} gives that $f(n)\geq \lfloor (\sqrt{8n-15}-3)/2\rfloor$, with equality for $2\leq n\leq 16$.
Since then Lai has obtained a series of sequential improvements on the lower bound (see \cite{BCFY,Lai} for details),
including the current record \cite{Lai} that $f(n)\geq \sqrt{238n/99}\approx 1.55\sqrt{n}$.
For the upper bound, Lai \cite{Lai90} proved $f(n)=O(\sqrt{n\log n})$, which was later reproved in \cite{CLJS}.
In a breakthrough result, Boros, Caro, F\"uredi and Yuster \cite{BCFY} deduced $f(n)\leq 1.98\sqrt{n}$ from the minimum cover of non-uniform hypergraphs,
and thus established the order of the magnitude of $f(n)$ to be $\Theta(\sqrt{n})$.
It remains open to determine $f(n)$, even asymptotically.

Another interesting problem is to consider the restricted version of Erd\H{o}s' problem for 2-connected graphs.
Following the notation in \cite{BCFY}, let $n+f_2(n)$ be the maximum number of edges in an $n$-vertex 2-connected graph in which no two cycles have the same length.
In 1988, employing the standard ear-decomposition of 2-connected graphs,
Shi \cite{shi} proved $f_2(n)\leq \sqrt{2n}+o(\sqrt{n})$.
In \cite{CLJS} Chen, Lehel, Jacobson and Shreve revisited this upper bound and used it to derive $f(n)=O(\sqrt{n\log n})$.
Using Sidon sequences in number theory,
they \cite{CLJS} also showed that $f_2(n)\geq \sqrt{n/2}-o(\sqrt{n})$.
A sequence of integers $a_1,a_2,...,a_k$ is called a {\it Sidon sequence} if all pairwise sums $a_i+a_j$ for $1\leq i\leq j\leq k$ are distinct.
Let $b_2(n)$ denote the maximum size of a Sidon subsequence of $\{1,2,...,n\}$.
It is well known that $b_2(n)=\sqrt{n}+o(\sqrt{n})$,
where the upper bound was proved by Erd\H{o}s and Tur\'an in their celebrated paper \cite{ET} (later simplified in \cite{Lin})
and the lower bound was provided by Singer \cite{Sing}.
Boros, Caro, F\"uredi and Yuster \cite{BCFY} refined the use of Sidon sequences and made a significant improvement on the lower bound of \cite{CLJS} by showing that
\begin{align}\label{equ:f2(n)}
f_2(n)\geq \sqrt{n}-o(\sqrt{n}).
\end{align}
To illustrate this somehow surprised relation between $f_2(n)$ and Sidon sequences,
we now give a sketch for the proof of \eqref{equ:f2(n)}.
Utilizing the result of Singer \cite{Sing} (together with Erd\H{o}s-Tur\'an Theorem \cite{ET}),
it is demonstrated in \cite{BCFY} that for any integer $n>0$,
there exist integers $a_1=1<a_2<...<a_k=n-1$ such that $k=\sqrt{n}-O(n^{9/20})$ and all differences $a_j-a_i$ for $1\leq i<j\leq k$ are pairwise distinct.
Construct an $n$-vertex 2-connected graph $G$ as follows:
let $V(G)=\{v_0,v_1,...,v_{n-1}\}$ and $E(G)$ consist of the edges in a Hamilton cycle $C=v_0v_1...v_{n-1}v_0$ and the edges $v_0v_{a_i}$ for all $1<i<k$.
It is easy to see that each cycle in $G$ contains two edges incident to $v_0$ (say $v_0v_{a_i}$ and $v_0v_{a_j}$) and the subpath of $C$ between $v_{a_i}$ and $v_{a_j}$ not containing $v_0$.
So all cycle lengths in $G$ are of the form $a_j-a_i+2$ for $1\leq i<j\leq k$, which are pairwise distinct.
This proves \eqref{equ:f2(n)} that $f_2(n)\geq e(G)-n=k-2=\sqrt{n}-o(\sqrt{n})$.

The authors of \cite{BCFY} further conjectured that the lower bound \eqref{equ:f2(n)} is asymptotically tight.
\begin{conjecture}[Boros, Caro, F\"uredi and Yuster, Conjecture 5.3 in \cite{BCFY}]\label{conj:f2(n)}
$$\lim_{n\to \infty} f_2(n)/\sqrt{n} =1.$$
\end{conjecture}
\noindent As remarked in \cite{BCFY}, this would imply ``the (difficult) upper bound in the Erd\H{o}s Tur\'an Theorem'' on Sidon sequences.
A weaker question was raised in \cite{Mar} to determine the maximum number of edges in a hamiltonian graph with no two cycles of the same length.\footnote{Note that a hamiltonian graph is naturally 2-connected.}

Our main result in this paper is to give a proof of Conjecture~\ref{conj:f2(n)} by the following.

\begin{thm}\label{thm: main}
Any $n$-vertex 2-connected graph with no two cycles of the same length contains at most $n+\sqrt{n}+o(\sqrt{n})$ edges.
\end{thm}


We introduce some notation.
Let $G$ be a graph.
For a subset $A$ of edges (or vertices) in $G$, let $G\bh A$ be the graph obtained from $G$ by deleting the elements in $A$.
Let $P$ be a path with endpoints $x,y$.
We say that $P$ is an $(x,y)$-path and any vertex or edge in $P\bh\{x,y\}$ is {\it inner}.
For a tree $T$ with $x,y\in V(T)$, we denote $xTy$ to be the unique subpath in $T$ between $x$ and $y$.
Suppose $F, F'$ are subgraphs of $G$.
By $F\dif F'$ we denote the subgraph consisting of the edges which appears in exactly one of $F$ and $F'$,
and by $F\bh F'$ we denote the subgraph consisting of the edges in $F$ but not in $F'$.
An \emph{$F$-ear} in $G$ is a path in $G$ whose two endpoints lie in $F$ but whose inner vertices do not.
An {\it ear-decomposition} of $G$ is a nested sequence $(G_0,G_1,...,G_s)$ of 
subgraphs of $G$ such that $G_0$ is a cycle, $G_{i+1}=G_i\cup P_{i+1}$ where $P_{i+1}$ is a $G_i$-ear in $G$ for $0\leq i<s$, and $G_s=G$.
It is well known that a graph $G$ is 2-connected if and only of it has an ear-decomposition.
Let us point out that any ear-decomposition of $G$ has $s=|E(G)|-|V(G)|$.
Throughout this paper, let $[n]=\{1,2,...,n\}$ and all logarithms in this paper are binary (with base $2$).

We organize this paper as follows.
In Section~2, we set up our proof environment by defining an ear-decomposition associated with a special linear ordering of vertices and
then using it to construct a family $\dF$ of paths which will serve as the building blocks to generate cycles later.
In Section~3, we prove some preliminary propositions on the paths in $\dF$ and classify all pairs of $\dF$ in three types.
In Section~4, we prove Lemma~\ref{lem:normal}, which gives a structural description on $\dF$ (very loosely speaking, it shows that almost all pairs of $\dF$ form a similar local structure).
In Section~5, we prove Lemma~\ref{lem: consecutive}, which roughly says that one can reorder the paths of $\dF$ in a nice way such that for almost every edge $e$, the paths containing $e$ are listed almost consecutively.
In Section~6, we complete the proof of Theorem~\ref{thm: main}.
In the final section, we conclude this paper by mentioning some remarks.

\section{Basic setting}
Throughout the rest of the paper, let $G$ be an $n$-vertex 2-connected graph with $n+s$ edges,
where $n$ is sufficiently large and $s\geq (1+o(1))\sqrt{n}$.
Our ultimate goal is to show that $G$ contains two cycles of the same length.
To this end, we assume in the rest of the paper that
$G$ contains at most one cycle of length $i$ for each $3\leq i\leq n$ and thus in particular,
\begin{align}\label{equ:(n-2)cycles}
\mbox{ $G$ contain at most $n-2$ cycles. }
\end{align}

To begin with, we define an ear-decomposition $(G_0,G_1,...,G_s)$ of $G$ and a linear order $\prec$ of $V(G)$ using the following iterated procedure.
(This will be crucial for all the coming proofs.)
\begin{itemize}
\item [(i)] Fix an edge $uv\in E(G)$ and let $G_0$ be any cycle in $G$ containing $uv$. Let $P_0=ux_1\cdots x_av$ be the path $G_0\bh\{uv\}$.
We define a linear order on $V(G_0)$ by letting $u\<x_1\<\cdots\<x_a\<v$.
\item [(ii)] Now suppose we have defined $G_{i-1}$ and a linear order $\prec$ on $V(G_{i-1})$ for some $1\leq i\leq s$.
Among all choices of $G_{i-1}$-ears in $G$, let $P_i$ be a $G_{i-1}$-ear with endpoints $\ell_i, r_i\in V(G_i)$ such that
$\ell_i$ is minimum under $\<$ of $V(G_{i-1})$ and subject to this, $r_i$ is minimum under $\<$ of $V(G_{i-1})$.\footnote{Note that by this choice, we have $\ell_i\< r_i$ for each $i\in [s]$.}
Let $G_i=G_{i-1}\cup P_i$.
Write $P_i=\ell_i y_1\cdots y_b r_i$ and let $\ell_i^+$ be the vertex of $G_{i-1}$ that succeeds $\ell_i$ immediately in the linear order $\<$.
We extend the linear order $\<$ on $V(G_{i-1})$ to $V(G_i)$
by inserting all vertices $y_j$ with $j\in [b]$ between $\ell_i$ and $\ell_i^+$ such that $\ell_i\<y_1\<\cdots\<y_b\<\ell_i^+$.
\end{itemize}

\noindent Using this ear-decomposition, we define
\begin{align}\label{equ:LR}
\mbox{$L=P_0\cup \left(\cup_{i\in [s]} P_i\bh \{r_i\}\right)$ ~~ and ~~ $R=P_0\cup \left(\cup_{i\in [s]} P_i\bh \{\ell_i\}\right)$.}
\end{align}
It is easy to see that $L$ and $R$ are two spanning trees in $G$, and we will view $u$ as the root of $L$ and $v$ as the root of $R$.
Now we define a family of $(u,v)$-paths as following:
$$\mbox{Let $f_0=P_0$ and for $i\in [s]$, let $L_i=uL\ell_i, ~R_i=r_iRv$ and $f_i=L_i\cup P_i\cup R_i$.}$$
Let $\mathcal{F}=\{f_i: 0\leq i\leq s\}$. These paths will be used to generate cycles in coming proofs.

\section{Preliminaries on $\mathcal{F}$}
In this section, we prove some basic propositions about the paths in $\dF$.
The first one can be derived directly from the above definitions.

\begin{prop} \label{prop:pre-suc}
Let $x,y\in V(G)$ and let $i, j$ be the minimum indices such that $x\in V(P_i)$ and $y\in V(P_j)$.
If $x\in uLy$ or $x\in yRv$, then $i\leq j$.
\end{prop}

\begin{prop}\label{prop:LiRj}
For any $i,j\in [s]$, $L_i\cup R_j$ does not contain cycles.
\end{prop}
\begin{proof}
    Suppose there is a cycle $C$ in $L_i\cup R_j$.
    Then there must exist two vertices $a,b$ such that $C=aL_ib\cup aR_jb$.
    Let $a\< b$ and let $k$ and $\ell$ be the minimum indices satisfying $a\in V(P_k)$ and $b\in V(P_\ell)$.
    Clearly we have $a\in uLb$ and then Proposition~\ref{prop:pre-suc} implies $k\leq \ell$.
    But we also have $b\in aRv$ and by Proposition~\ref{prop:pre-suc} again, we derive $\ell\leq k$.
    So $k=\ell$ and $a,b$ are two inner vertices of $P_k$.
    This shows that $aL_ib=aLb=aP_kb=aRb=aR_jb$, a contradiction.
\end{proof}

\begin{prop}\label{prop:P-notin-f}
For distinct $i, j\in \{0,1,...,s\}$, we have $P_i\not\subseteq f_j$.
\end{prop}
\begin{proof}
Suppose that $P_i$ is a subpath of $f_j$ for some $i\neq j$. Clearly we may assume $i,j\in [s]$.
Since $P_i,P_j$ have no common edges, it follows that either $P_i\subseteq L_j$ or $P_i\subseteq R_j$.
Now we note that $P_i$ is not a subpath of the tree $L$ (or respectively $R$), but $L_j$ (or respectively $R_j$) is, a contradiction.
\end{proof}

The following proposition is also straightforward to see (we omit its proof here).

\begin{prop}\label{prop:orient-path}
For any $i\in \{0,1,...,s\}$, $f_i$ is a $(u,v)$-path,
whose vertices, as traversing from $u$ to $v$, increase in the linear order $\<$.\footnote{It will be convenient for us to picture that each $f_i$ has an imagined orientation from $u$ to $v$, as to capture the linear ordering $\<$ on its vertices.}
\end{prop}
\noindent

Let $i\neq j\in \{0,1,...,s\}$. A vertex in $f_i\cup f_j$ is called {\bf splitting} if it has at least three neighbors in $f_i\cup f_j$.
By Proposition~\ref{prop:orient-path}, there exists some integer $t\geq 1$ such that $f_i\dif f_j$ consists of $t$ cycles $a_\ell P_i b_\ell \cup a_\ell P_j b_\ell$ for $\ell\in [t]$, where $a_1\<b_1\pq a_2\<\cdots \<b_{t-1}\pq a_t\<b_t$ are all splitting vertices of $f_i\cup f_j$.

\begin{prop} \label{prop:diff1}
For distinct $i, j\in \{0,1,...,s\}$, $f_i\dif f_j$ consists of one or two cycles such that
each cycle shares edges with $P_i\cup P_j$ and each of $P_i$ and $P_j$ shares edges with at most one cycle in $f_i\dif f_j$.
\end{prop}
\begin{proof}
First we show that each cycle $C$ in $f_i\dif f_j$ contains some edges in $P_i\cup P_j$.
Suppose not. Then we have $E(C)\subseteq E(L_i\cup R_i\cup L_j\cup R_j)$.
Let us assume $i<j$ here. By the choice of ears, we see $\ell_i\pq \ell_j$.
If there exists $x\in V(L_i)\cap V(R_j)$, then $\ell_j\<r_j\pq x\pq \ell_i$, a contradiction.
So $V(L_i)\cap V(R_j)=\emptyset$.
If $C$ has an edge in $L_i$, then  $L_i$ has at least two splitting vertices of $f_i\cup f_j$.
But $L_i\cup L_j$ has at most one splitting vertex, while $L_i$ and $R_i\cup R_j$ have no common vertex, a contradiction.
So $C$ has no edge in $L_i$ and similarly we can show $C$ has no edge in $R_j$.
Thus $C$ is contained in $R_i\cup L_j$, a contradiction to Proposition~\ref{prop:LiRj}.

Next we show each of $P_i$ and $P_j$ shares edges with at most one cycle in $f_i\dif f_j$.
Suppose for a contradiction that $P_i$ shares edges with two cycles in $f_i\dif f_j$.
Then there exists a common vertex $x$ in $f_i$ and $f_j$ between these two cycles.
We see that $x$ is an inner vertex of $P_i$ and thus cannot be an inner vertex of $P_j$.
So $x\in V(L_j)\cup V(R_j)$.
If $x\in V(L_j)$, then both of $uf_ix$ and $uL_jx$ are in $L$ and thus cannot contain cycles, a contradiction.
If $x\in V(R_j)$, then $xf_iv\cup xR_jv\subseteq R$, a contradiction.
Combining with these two facts, it is easy to see that $f_i\dif f_j$ consists of at most two cycles.
\end{proof}

\begin{prop}\label{prop:2cycleI}
Suppose $f_i\dif f_j$ consists of two cycles, where $i<j$.
Let $a\< b\pq c\< d$ be all splitting vertices of $f_i\cup f_j$.
Then $P_i\subseteq af_ib$, $P_j\subseteq cf_jd$ and there exists some $k<i$ such that
$f_k=uf_jc\cup cf_iv\in\dF$ and $b,c$ are inner vertices of $P_k$.
\end{prop}
\begin{proof}
Let $k$ be the minimum index such that $c\in V(P_k)$. So $c$ is an inner vertex of $P_k$.
By Proposition~\ref{prop:diff1}, we have $E(P_i)\cap E(af_ib)\neq \emptyset$ and $E(P_j)\cap E(cf_jd)\neq\emptyset$
(this is because, otherwise, $E(P_i)\cap E(af_ib)=\emptyset$ and $E(P_j)\cap E(af_jb)\neq\emptyset$,
which would imply that $\ell_j\<b\pq\ell_i$, contradicting that $i<j$).
So $c\in V(L_j\cup P_j\bh r_j)$.
This, together with the fact $c$ is an inner vertex of $P_k$, show that $uf_kc=uLc=uf_jc$.
We also have $E(P_i)\cap E(cf_id)=\emptyset$ and thus $c\in V(R_i)$, implying $k<i$ and $cf_kv=cRv=cf_iv$.
Hence $f_k=uf_jc\cup cf_iv\in \dF$.

Now we see $f_i\dif f_k=af_ib\cup af_jb$ is a cycle and since $k<i$, this cycle must contain the entire $P_i$ and some edge in $P_k$.
Similarly since $k<j$, $f_j\dif f_k=cf_id\cup cf_jd$ is a cycle which contains the entire $P_j$ and some edge in $P_k$.
Therefore, we see $P_i\subseteq af_ib$, $P_j\subseteq cf_jd$, and $P_k$ contains some edge in $af_jb$ and some edge in $cf_id$,
which shows that $b,c$ are inner vertices of $P_k$.
\end{proof}

Let $i, j$ be distinct. By Proposition~\ref{prop:diff1}, we see that $f_i\backslash f_j$ consists of one or two subpaths,
exactly one of which contains edges in $P_i$.
By the {\bf primary segment} $\ps(i,j)$ of $f_i\backslash f_j$, we denote the unique subpath in $f_i\backslash f_j$ containing edges in $P_i$.
Note that $\ps(i,j)$ and $\ps(j,i)$ are distinct.

\begin{figure}[ht!]\label{fig:A}
	\begin{minipage}{0.5\linewidth}
		\centering
\begin{tikzpicture}[scale=0.8,line cap=round,line join=round,>=triangle 45,x=1cm,y=1cm]
\clip(-1,-1) rectangle (9,2);
\draw [line width=1pt] (0,0)-- (7,0);
\draw [shift={(2,0)},line width=1pt,color=red]  plot[domain=0:3.141592653589793,variable=\t]({1*1*cos(\t r)+0*1*sin(\t r)},{0*1*cos(\t r)+1*1*sin(\t r)});
\draw [shift={(5,0)},line width=1pt,color=blue]  plot[domain=0:3.141592653589793,variable=\t]({1*1*cos(\t r)+0*1*sin(\t r)},{0*1*cos(\t r)+1*1*sin(\t r)});

\begin{scriptsize}
\draw [fill=black] (0,0) circle (1.5pt);
\draw[color=black] (0 ,-0.25) node {$u$};
\draw [fill=black] (7,0) circle (1.5 pt);
\draw[color=black] (7 ,-0.25) node {$v$};

\draw [fill=black] (1,0) circle (1.5pt);
\draw[color=black] (1 ,-0.25) node {$a$};
\draw [fill=black] (3,0) circle (1.5pt);
\draw[color=black] (3 ,-0.25) node {$b$};
\draw[color=black] (2,1.25) node {$f_i$};
\draw [fill=black] (4,0) circle (1.5pt);
\draw[color=black] (4 ,-0.25) node {$c$};
\draw [fill=black] (6,0) circle (1.5pt);
\draw[color=black] (6 ,-0.25) node {$d$};
\draw[color=black] (5 ,1.25) node {$f_j$};
\draw[color=black] (7.5, 0) node {$f_k$};

\draw[color=black] (-0.1, 1.3) node {\large Type-\1};
\end{scriptsize}
\end{tikzpicture}
\end{minipage}	
\begin{minipage}{0.5\linewidth}
		\centering
\begin{tikzpicture}[scale=0.8,line cap=round,line join=round,>=triangle 45,x=1cm,y=1cm]
\clip(-1,-1) rectangle (9,2);
\draw [line width=1pt] (0,0)-- (7,0);
\draw [shift={(3,0)},line width=1pt,color=red]  plot[domain=0:3.141592653589793,variable=\t]({1*1*cos(\t r)+0*1*sin(\t r)},{0*1*cos(\t r)+1*1*sin(\t r)});
\draw [shift={(4,0)},line width=1pt,color=blue]  plot[domain=0:3.141592653589793,variable=\t]({1*1*cos(\t r)+0*1*sin(\t r)},{0*1*cos(\t r)+1*1*sin(\t r)});

\begin{scriptsize}
\draw [fill=black] (0,0) circle (1.5pt);
\draw[color=black] (0 ,-0.25) node {$u$};
\draw [fill=black] (7,0) circle (1.5 pt);
\draw[color=black] (7 ,-0.25) node {$v$};
\draw[color=black] (3.5, -0.25) node {$f_k$};
\draw [fill=black] (2,0) circle (1.5pt);
\draw[color=black] (2 ,-0.25) node {$a$};
\draw [fill=black] (3,0) circle (1.5pt);
\draw[color=black] (3 ,-0.25) node {$b$};
\draw[color=black] (3,1.25) node {$f_i$};
\draw [fill=black] (4,0) circle (1.5pt);
\draw[color=black] (4 ,-0.25) node {$c$};
\draw [fill=black] (5,0) circle (1.5pt);
\draw[color=black] (5 ,-0.25) node {$d$};
\draw[color=black] (4 ,1.25) node {$f_j$};
\draw[color=black] (7.5, 0) node {$f_k$};

\draw[color=black] (0.3, 1.3) node {\large Type-\2};
\end{scriptsize}
\end{tikzpicture}
\end{minipage}
\caption*{Figure 1}
\end{figure}

We now classify all pairs of $\dF$ in the following three types:
\begin{itemize}
\item A pair $\{f_i,f_j\}$ in $\dF$ is called {\bf type-\1}, if $f_i\dif f_j$ consists of two cycles.
In this case, we call the path $f_k\in \dF$ guaranteed in Proposition~\ref{prop:2cycleI} as the {\bf base} of $\{f_i,f_j\}$. See Figure 1 for an illustration.
\item A pair $\{f_i,f_j\}$ in $\dF$ is called {\bf type-\2}, if it is not type-\1 and there exists some $f_\ell\in \dF$
such that $\ps(i,\ell)=af_ic$ and $\ps(j,\ell)=bf_jd$ where $a\< b\< c\< d$ lie in $f_\ell$.
Such a path $f_\ell$ is called a {\bf crossing path} of $\{f_i,f_j\}$, and the crossing path $f_\ell$ with minimum $\ell$ is called the {\bf base} of $\{f_i,f_j\}$. See Figure 1.
\item Finally, a pair $\{f_i,f_j\}$ in $\dF$ is {\bf normal}, if it is neither type-\1 nor type-\2.
\end{itemize}

\begin{prop}\label{prop:degenerate-2}
Let $i,j,\ell$ be distinct. If $P_\ell\subseteq f_i\cup f_j$, then $\{f_i,f_j\}$ is type-\1 with base $f_\ell$.
\end{prop}
\begin{proof}
By Proposition~\ref{prop:P-notin-f}, we see $P_\ell \not\subseteq f_i$ and $P_\ell \not\subseteq f_j$.
So $P_\ell \cap f_i\neq \emptyset$ and $P_\ell \cap f_j\neq \emptyset$.
By Proposition~\ref{prop:2cycleI}, $\{f_i,f_\ell\}$ and $\{f_j,f_\ell\}$ are not type-\1.
Let $a\prec b$ be the splitting vertices in $f_i\cup f_\ell$ and $c\prec d$ be the splitting vertices in $f_j\cup f_\ell$.
We may assume $a\pq c$.
If $c\<b$, since $E(cf_\ell b)\cap E(f_i\cup f_j)=\emptyset$, we see $E(cf_\ell b)\cap E(P_\ell)=\emptyset$.
This implies that either $E(P_\ell)\cap E(af_\ell b)=\emptyset$ or $E(P_\ell)\cap E(cf_\ell d)=\emptyset$, a contradiction.
Hence, $a\< b\pq c\< d$ lie in $f_\ell$.
If there exists some $z\in V(af_ib)\cap V(cf_jd)$, then we have a contradiction that $c\<z\<b$.
So $af_ib$ and $cf_jd$ are internally disjoint.
Now we see that $\{f_i,f_j\}$ is a type-\1 pair with base $f_\ell$.
\end{proof}

For paths $R_1, R_2$ in $G$, we write $R_1\pq R_2$ (resp., $R_1\<R_2$) if any $s\in V(R_1)$ and $t\in V(R_2)$ satisfy $s\pq t$ (resp., $s\<t$).

\begin{prop}\label{prop:degenerate-3}
Let $i, j, k, \ell$ be distinct. If $P_\ell\subseteq f_i\cup f_j\cup f_k$, then there exist $\alpha,\beta\in \{i,j,k\}$ such that $P_\ell\subseteq f_\alpha\cup f_\beta$ and thus $\{f_\alpha,f_\beta\}$ is type-\1 with base $f_\ell$.
\end{prop}
\begin{proof}
If there are $\alpha,\beta\in \{i,j,k\}$ such that $P_\ell\subseteq f_\alpha\cup f_\beta$, then the conclusion follows by Proposition~\ref{prop:degenerate-2}.
So we may assume that there exist edges $e_i, e_j, e_k$ in $P_\ell$ such that $e_i\in f_i\bh (f_j\cup f_k)$, $e_j\in f_j\bh (f_i\cup f_k)$, and $e_k\in f_k\bh (f_i\cup f_j)$.
Without loss of generality, we may assume that $e_i\pq e_j\pq e_k$ lie in $f_\ell$.
Clearly $\{f_j,f_\ell\}$ is not type-\1 (as otherwise, $E(P_\ell)\cap E(f_j)=\emptyset$ by Proposition~\ref{prop:2cycleI}).
We have $e_j\in E(f_j\cap f_\ell)$.
So $f_j$ contains either the subpath of $f_\ell$ from $u$ to $e_j$ or the subpath of $f_\ell$ from $e_j$ to $v$,
which implies either $e_i\in E(f_j)$ or $e_k\in E(f_j)$, a contradiction. This completes the proof.
\end{proof}

\section{Almost all pairs are normal}\label{sec:normal-pairs}

For any subset $\dA$ of $\dF$,
we define $G(\dA)$ to be the subgraph of $G$ consisting of all edges $e$, which appears in some path of $\dA$ but not in every path of $\dA$.
Note that if $\dA\subseteq \dB\subseteq \dF$, then $G(\dA)$ is a subgraph of $G(\dB)$.
We say $\{x,y\}\subseteq \bigcap_{f_i\in \dA} V(f_i)$ is the {\bf separator} of $\dA$, if $G(\dA)\subseteq \bigcup_{f_i\in \dA} xf_iy$ and subject to this, $\bigcup_{f_i\in \dA} xf_iy$ is minimal.

The main result of this section is the following lemma.
\begin{lem}\label{lem:normal}
There exist disjoint subsets $\dF_1, \dF_2, \dF_3, \dF_4$ in $\dF$ such that $\sum_{i\in [4]} |\dF_i|\geq s-90\sqrt{n}/\log n$,
all $G(\dF_i)$'s are edge-disjoint, and each $\dF_i$ contains at most $2\sqrt{n}\log^2 n$ pairs of type-\1 and type-\2.
\end{lem}

We show that the proof of Lemma \ref{lem:normal} can be reduced to two lemmas in below.

\begin{lem} \label{lem: type 1}
    There exist two disjoint subsets $\dA_1,\dA_2$ in $\dF$ such that $|\dA_1|+|\dA_2|>s-30\sqrt{n}/\log n$,
    $G(\dA_1)$ and $G(\dA_2)$ are edge-disjoint, and each $\dA_i$ contains at most $\sqrt{n}\log^2n$ type-\1 pairs.
\end{lem}

\begin{lem}\label{lem: type 2}
    There exist two disjoint subsets $\dB_1,\dB_2$ in $\dF$ such that $|\dB_1|+|\dB_2|>s-60\sqrt{n}/\log n$,
    $G(\dB_1)$ and $G(\dB_2)$ are edge-disjoint, and each $\dB_i$ contains at most $\sqrt{n}\log^2n$ type-\2 pairs.
\end{lem}

\medskip

{\noindent \bf Proof of Lemma \ref{lem:normal} (Assume Lemmas \ref{lem: type 1} and \ref{lem: type 2}).}
Let $\dA_1,\dA_2$ be obtained from Lemma \ref{lem: type 1} and $\dB_1,\dB_2$ be obtained from Lemma \ref{lem: type 2}.
For $1\leq i, j\leq 2$, let $\dC_{ij}=\dA_i\cap \dB_j$.
Following the properties of $\dA_1,\dA_2, \dB_1, \dB_2$, it is easy to see that the four subsets $\dC_{ij}$'s are disjoint,
each $\dC_{ij}$ contains at most $2\sqrt{n}\log^2 n$ pairs of type-\1 and type-\2,
and $$\sum_{1\leq i, j\leq 2} |\dC_{ij}|=|(\dA_1\cup \dA_2)\cap (\dB_1\cup \dB_2)|\geq |\dA_1\cup \dA_2|+|\dB_1\cup \dB_2|-|\dF|\geq s-90\sqrt{n}/\log n.$$
It remains to show that all $G(\dC_{ij})$'s are edge-disjoint.
Fix $i\in [2]$. Since $\dC_{ij}\subseteq \dB_j$ for each $j\in [2]$, by the observation before Lemma \ref{lem:normal},
we have $G(\dC_{ij})\subseteq G(\dB_j)$ for each $j\in [2]$.
We also see from Lemma \ref{lem: type 2} that $G(\dB_1)$ and $G(\dB_2)$ are edge-disjoint,
so it is clear that $G(\dC_{i1})$ and $G(\dC_{i2})$ are edge-disjoint for all $i\in [2]$.
Similarly, $G(\dC_{1j})$ and $G(\dC_{2j})$ are edge-disjoint for all $j\in [2]$, finishing the proof. \qed

\medskip

For the proofs of Lemmas \ref{lem: type 1} and \ref{lem: type 2} (also for the proof in Section~\ref{sec:reorder-dF}),
we need to introduce some notation on collections of paths in $\dF$, which are used to generate cycles of some special characters.
\begin{itemize}
\item A triple $\{f_i,f_j,f_\ell\}$ in $\dF$ is called {\bf feasible}, if $\{uv\}\cup f_i\cup f_j\cup f_\ell$ contains a cycle $C(i,j,\ell)$
such that for all possible $\{\alpha,\beta,\gamma\}=\{i,j,\ell\}$, either $P_\alpha \backslash (f_\beta\cup f_\gamma)=\emptyset$ or $C(i,j,\ell)$ contains some edge in $P_\alpha \backslash (f_\beta\cup f_\gamma)$.
Such a cycle $C(i,j,\ell)$ is called {\bf 3-feasible}.

\item A quadruple $\{f_i,f_j,f_k,f_\ell\}$ in $\dF$ is called {\bf feasible}, if $\{uv\}\cup f_i\cup f_j\cup f_k\cup f_\ell$ contains a cycle $C(i,j,k,\ell)$
such that for all possible $\{\alpha,\beta,\gamma,\theta\}=\{i,j,k,\ell\}$,
either $P_\alpha \backslash (f_\beta\cup f_\gamma\cup f_\theta)=\emptyset$ or $C(i,j,k,\ell)$ contains some edge in $P_\alpha \backslash (f_\beta\cup f_\gamma\cup f_\theta)$.
Such a cycle $C(i,j,k,\ell)$ is called {\bf 4-feasible}.
\end{itemize}

\subsection{The number of feasible tuples}\label{subsec:feasible-cycles}

In the following proposition, we estimate the number of feasible triples and quadruples in $\dF$.

\begin{prop}\label{prop:feasible}
There are at most $ n$ feasible triples and at most $4n$ feasible quadruples in $\dF$.
\end{prop}
\begin{proof}
A feasible triple $\{f_i,f_j,f_k\}$ is called {\it degenerate},
if there exists some $\{\alpha,\beta,\gamma\}=\{i,j,k\}$ with $P_\alpha \backslash (f_\beta\cup f_\gamma )=\emptyset$.
By Propositions \ref{prop:2cycleI} and \ref{prop:degenerate-2}, $\{f_\beta, f_\gamma\}$ is type-\1 with base $f_\alpha$, where $\alpha<\min\{\beta, \gamma\}$.
Thus, each degenerate feasible triple consists of a unique type-\1 pair and its base.\footnote{Reversely, each type-\1 pair determines a degenerate feasible triple.}
Next we claim that for any two distinct non-degenerate feasible triples $\{f_i,f_j,f_k\}$ and $\{f_{i'},f_{j'},f_{k'}\}$,
their 3-feasible cycles $C(i,j,k)$ and $C(i',j',k')$ are distinct.
Suppose $C(i,j,k)=C(i',j',k')$.
Let $i<j<k$ and $i'<j'<k'$.
By symmetry let us assume $k\geq k'$.
If $k>k'$, then we have $E(P_k)\cap E(f_{i'}\cup f_{j'}\cup f_{k'})=\emptyset$,
which contradicts that $E(P_k)\cap E(C(i,j,k))\neq \emptyset$.
Thus it must be $k=k'$.
Now assume $j\geq j'$.
If $j>j'$, then $E(P_j)\cap E(f_{i'}\cup f_{j'})=\emptyset$,
implying that $C(i',j',k)$ does not contain edges in $P_j\bh f_k$.
But $C(i,j,k)$ does contain edges in $P_j\backslash f_k$, a contradiction.
Thus $j=j'$.
Finally we may assume $i\geq i'$.
If $i>i'$, then $E(P_i)\cap E(f_{i'})=\emptyset$ and thus $C(i',j,k)$ does not contain edges in $P_i\bh(f_k\cup f_j)$.
However, $C(i,j,k)$ does contain such edges.
This gives a contradiction that $\{i,j,k\}=\{i',j',k'\}$, proving the claim.
It is straightforward to see that each non-degenerate feasible triple and each degenerate feasible triple have different 3-feasible cycles.
Hence, each feasible triple contributes a unique 3-feasible cycle.
By \eqref{equ:(n-2)cycles}, we see that $\dF$ has at most $n$ feasible triples.

Similarly, a feasible quadruple $\{f_i,f_j,f_k,f_\ell\}$ is called {\it degenerate}
if there exists some $\{\alpha,\beta,\gamma, \theta\}=\{i,j,k,\ell\}$ with $P_\alpha \backslash (f_\beta\cup f_\gamma\cup f_\theta)=\emptyset$.
It is analogous to show that each non-degenerate feasible quadruple contributes a unique 4-feasible cycle.
So $\dF$ has at most $n$ non-degenerate feasible quadruples.

Now consider a degenerate feasible quadruple $\{f_i,f_j,f_k,f_\ell\}$ with $P_i\subseteq f_j\cup f_k\cup f_\ell$.

We claim that $P_\alpha \backslash (f_\beta\cup f_\gamma\cup f_i)\neq \emptyset$ for any $\{\alpha,\beta,\gamma\}=\{j,k,\ell\}$.
Suppose for a contradiction that $P_j\subseteq f_i\cup f_k\cup f_\ell$.
Since both of $f_i, f_j$ cannot be the base of $\{f_k,f_\ell\}$ at the same time,
by the symmetries between $f_i$ and $f_j$ and between $f_k$ and $f_\ell$,
using Proposition~\ref{prop:degenerate-3} we may assume that $f_i$ is the base of the type-\1 pair $\{f_j,f_k\}$.
We then see $i<\min \{j,k\}$ from Proposition~\ref{prop:2cycleI}, and this in turn implies that
$\{f_k,f_\ell\}$ is type-\1 with base $f_j$.
But then $\{f_j,f_k\}$ is not type-\1, a contradiction.

By Proposition~\ref{prop:degenerate-3}, we may assume that $P_i\subseteq f_j\cup f_k$ and $\{f_j, f_k\}$ is type-\1 with base $f_i$.
This yields $f_i\subseteq f_j\cup f_k$ and thus the 4-feasible cycle $C(i,j,k,\ell)$ is contained in $\{uv\}\cup f_j\cup f_k\cup f_\ell$.
Using the previous claim, $C(i,j,k,\ell)$ contains some edge in $P_\alpha\bh (f_\beta\cup f_\gamma\cup f_i)\subseteq P_\alpha\bh (f_\beta\cup f_\gamma)$ for any $\{\alpha, \beta, \gamma\}=\{j,k,\ell\}$.
Thus this cycle is also 3-feasible for $\{f_j,f_k,f_\ell\}$, which now is known to be a feasible triple.
Note that $f_i$ is the base of $\{f_j, f_k\}$.
So the number of degenerate feasible quadruples is at most three times the number of feasible triples, that is at most $3n$.
Hence in total there are at most $4n$ feasible quadruples in $\dF$.
This completes the proof.
\end{proof}

For any pair $\{i,j\}\subseteq \{0\}\cup [s]$, let $\dW_{ij}$ be the set of all paths $f_\ell\in \dF$ such that the triple $\{f_i,f_j,f_\ell\}$ is either feasible or contained in a feasible quadruple.
By Proposition~\ref{prop:feasible}, we see that
\begin{align}\label{equ:Wij}
\sum_{\mbox{all pairs } \{i,j\}} |\dW_{ij}|\leq  n \binom{3}{2}+4n \binom{4}{2}\cdot 2=51n.
\end{align}

\subsection{Proof of Lemma \ref{lem: type 1}}\label{Subsec:typeI}
To show Lemma \ref{lem: type 1}, we will first establish some properties on type-\1 pairs.
In this subsection, unless otherwise specified we assume $i<j$ and $\{f_i,f_j\}$ is a type-\1 pair in $\dF$ with base $f_k$.
Let $a\prec b\pq c\prec d$ be all splitting vertices in $f_i\cup f_j$.
For any $\ell\in [s]\bh \{i,j,k\}$, by Proposition~\ref{prop:degenerate-2} we have $P_\ell\not\subseteq f_i\cup f_j$.

Let $\dM_{ij}$ consist of paths $f_\ell\in \dF$ with $\ell\notin \{i,j,k\}$ such that $f_\ell\bh (f_i\cup f_j)$ is a path $x_\ell f_\ell y_\ell$ with $x_\ell\< y_\ell\pq c$ and $cf_\ell v=cf_iv$,
where either $x_\ell\pq a\< b\pq y_\ell \pq c$ or both $x_\ell, y_\ell$ lie in one of $af_ib$ and $af_jb$.
Let $\dN_{ij}$ consist of paths $f_\ell\in \dF$ with $\ell\notin \{i,j,k\}$ such that $f_\ell\bh (f_i\cup f_j)$ is a path $x_\ell f_\ell y_\ell$ with $b\pq x_\ell\< y_\ell$ and $uf_\ell b=uf_jb$,
where either $b\pq x_\ell\pq c\< d\pq y_\ell$ or both $x_\ell, y_\ell$ lie in one of $cf_id$ and $cf_jd$.

\begin{prop}\label{prop:non-feasible-triple-typeI}
Let $\ell\notin \{i,j,k\}$. If the triple $\{f_i,f_j,f_\ell\}$ is not feasible, then $f_\ell \in \dM_{ij}\cup \dN_{ij}$.
\end{prop}
\begin{proof}
Fix $\ell\notin \{i,j,k\}$ such that $\{f_i,f_j,f_\ell\}$ is a non-feasible triple.
We point out that by Proposition~\ref{prop:degenerate-2} and the fact that $\{f_i, f_j\}$ is type-\1,
$P_\alpha\bh (f_\beta\cup f_\gamma)\neq \emptyset$ for any $\{\alpha, \beta, \gamma\}=\{i, j, \ell\}$.
Take any subpath $xf_\ell y$ in $f_\ell\bh (f_i\cup f_j)$ which contains some edge of $P_\ell$, where we assume $x\<y$.
By Proposition~\ref{prop:2cycleI}, since $i<j$, we see that $P_i\subseteq af_ib$, $P_j\subseteq cf_jd$,
and $f_i\cup f_j$ can be partitioned into edge-disjoint paths $af_ib, cf_jd$ and $f_k$.
We will proceed by considering whether $x, y$ lie in $af_ib, cf_jd$ or $f_k$.
In the coming proof, we will make use of the symmetry between $x$ and $y$ and the symmetry between $f_i$ and $f_j$.

\begin{figure}[ht!]\label{fig:A}
\begin{minipage}{0.33 \linewidth}
        \centering
\begin{tikzpicture}[scale=0.7,line cap=round,line join=round,>=triangle 45,x=1cm,y=1cm]
\clip(-0.75,-0.5) rectangle (7.75,2.5);
\draw [line width=1pt] (0,0)-- (7,0);
\draw [shift={(2,0)},line width=1pt,color=red]  plot[domain=0:3.141592653589793,variable=\t]({1*1*cos(\t r)+0*1*sin(\t r)},{0*1*cos(\t r)+1*1*sin(\t r)});
\draw [shift={(5,0)},line width=1pt,color=blue]  plot[domain=0:3.141592653589793,variable=\t]({1*1*cos(\t r)+0*1*sin(\t r)},{0*1*cos(\t r)+1*1*sin(\t r)});
\draw[line width=0.5 pt] (2,1) to[out=40,in=140] (5,1);

\begin{scriptsize}
\draw [fill=black] (2,1) circle (1.5pt);
\draw[color=black] (1.8 ,1.25) node {$x$};
\draw [fill=black] (5,1) circle (1.5 pt);
\draw[color=black] (5.2 ,1.25) node {$y$};
\draw[color=black] (3.5 ,2) node {$f_{\ell}$};

\draw [fill=black] (0,0) circle (1.5pt);
\draw[color=black] (0 ,-0.25) node {$u$};
\draw [fill=black] (7,0) circle (1.5 pt);
\draw[color=black] (7 ,-0.25) node {$v$};
\draw[color=black] (2, -0.3) node {$f_k$};
\draw [fill=black] (1,0) circle (1.5pt);
\draw[color=black] (1 ,-0.25) node {$a$};
\draw [fill=black] (3,0) circle (1.5pt);
\draw[color=black] (3 ,-0.25) node {$b$};
\draw[color=black] (3 ,0.8 ) node {$f_i$};
\draw [fill=black] (4,0) circle (1.5pt);
\draw[color=black] (4 ,-0.25) node {$c$};
\draw [fill=black] (6,0) circle (1.5pt);
\draw[color=black] (6 ,-0.25) node {$d$};
\draw[color=black] (6 ,0.8) node {$f_j$};
\draw[color=black] (6.5, 0.3) node {$f_k$};
\end{scriptsize}
\end{tikzpicture}
\caption*{(a)}
\end{minipage}
\begin{minipage}{0.33 \linewidth}
        \centering
\begin{tikzpicture}[scale=0.7,line cap=round,line join=round,>=triangle 45,x=1cm,y=1cm]
\clip(-0.75,-0.5) rectangle (7.75,2.5);
\draw [line width=1pt] (0,0)-- (7,0);
\draw [shift={(2,0)},line width=1pt,color=red]  plot[domain=0:3.141592653589793,variable=\t]({1*1*cos(\t r)+0*1*sin(\t r)},{0*1*cos(\t r)+1*1*sin(\t r)});
\draw [shift={(5,0)},line width=1pt,color=blue]  plot[domain=0:3.141592653589793,variable=\t]({1*1*cos(\t r)+0*1*sin(\t r)},{0*1*cos(\t r)+1*1*sin(\t r)});
\draw [line width=0.5pt] (0.5,0 ) to[out=90,in=160]  (2,1);

\begin{scriptsize}
\draw [fill=black] (2,1) circle (1.5pt);
\draw[color=black] (2.2 ,1.25) node {$x'$};
\draw [fill=black] (0.5,0 ) circle (1.5 pt);
\draw[color=black] (0.5,-0.25 ) node {$y'$};
\draw[color=black] (0.5 ,0.9) node {$f_{\ell}$};

\draw [fill=black] (0,0) circle (1.5pt);
\draw[color=black] (0 ,-0.25) node {$u$};
\draw [fill=black] (7,0) circle (1.5 pt);
\draw[color=black] (7 ,-0.25) node {$v$};
\draw [fill=black] (1,0) circle (1.5pt);
\draw[color=black] (1 ,-0.25) node {$a$};
\draw [fill=black] (3,0) circle (1.5pt);
\draw[color=black] (3 ,-0.25) node {$b$};
\draw[color=black] (3 ,0.8 ) node {$f_i$};
\draw [fill=black] (4,0) circle (1.5pt);
\draw[color=black] (4 ,-0.25) node {$c$};
\draw [fill=black] (6,0) circle (1.5pt);
\draw[color=black] (6 ,-0.25) node {$d$};
\draw[color=black] (6  ,0.8) node {$f_j$};
\draw[color=black] (6.5, 0.3) node {$f_k$};
\end{scriptsize}
\end{tikzpicture}
\caption*{(b)}
\end{minipage}
\begin{minipage}{0.33 \linewidth}
        \centering
\begin{tikzpicture}[scale=0.7,line cap=round,line join=round,>=triangle 45,x=1cm,y=1cm]
\clip(-0.75,-0.5) rectangle (7.75,2.5);
\draw [line width=1pt] (0,0)-- (7,0);
\draw [shift={(2,0)},line width=1pt,color=red]  plot[domain=0:3.141592653589793,variable=\t]({1*1*cos(\t r)+0*1*sin(\t r)},{0*1*cos(\t r)+1*1*sin(\t r)});
\draw [shift={(5,0)},line width=1pt,color=blue]  plot[domain=0:3.141592653589793,variable=\t]({1*1*cos(\t r)+0*1*sin(\t r)},{0*1*cos(\t r)+1*1*sin(\t r)});
\draw[line width=0.5pt] (1.4,0.8).. controls (1.5,1.5) and  ( 2.5,1.5) .. (2.6,0.8);
\begin{scriptsize}
\draw [fill=black] (1.4,0.8) circle (1.5pt);
\draw [fill=black] (1.2,1 ) node {$x$};
\draw [fill=black] (2.6,0.8) circle (1.5pt);
\draw[color=black] (2.8,1) node {$y$};
\draw[color=black] (2 ,1.8) node {$f_{\ell}$};

\draw [fill=black] (0,0) circle (1.5pt);
\draw[color=black] (0 ,-0.25) node {$u$};
\draw [fill=black] (7,0) circle (1.5 pt);
\draw[color=black] (7 ,-0.25) node {$v$};
\draw [fill=black] (1,0) circle (1.5pt);
\draw[color=black] (1 ,-0.25) node {$a$};
\draw [fill=black] (3,0) circle (1.5pt);
\draw[color=black] (3 ,-0.25) node {$b$};
\draw[color=black] (2,0.7) node {$f_i$};
\draw [fill=black] (4,0) circle (1.5pt);
\draw[color=black] (4 ,-0.25) node {$c$};
\draw [fill=black] (6,0) circle (1.5pt);
\draw[color=black] (6 ,-0.25) node {$d$};
\draw[color=black] (5 ,1.3) node {$f_j$};
\draw[color=black] (6.5, 0.3) node {$f_k$};
\end{scriptsize}
\end{tikzpicture}
\caption*{(c)}
\end{minipage}
\caption*{Figure 2}
\end{figure}
First, suppose that $x,y$ are not both in $af_ib$, $cf_jd$ or $f_k$.
If $x$ lies in $af_ib$ and $y$ lies in $cf_jd$ (see Figure 2-a),
then as $P_i\subseteq af_ib$ and $P_j\subseteq cf_jd$, one can see that $\{uv\}\cup f_i\cup f_j\cup xf_\ell y$ always contains a 3-feasible cycle $C(i,j,\ell)$ and thus  $\{f_i,f_j,f_\ell\}$ is feasible, a contradiction.
By symmetry between $af_ib$ and $cf_jd$, it suffices to consider the case that one of $\{x,y\}$ (say $x'$) lies in $af_ib\bh\{a,b\}$ and the other $y'$ lies in $f_k\bh\{a,b\}$  (see Figure 2-b),
where $\{x,y\}=\{x',y'\}$ and $y'$ may lies in any of $uf_ia, af_jb, bf_ic, cf_id$ and $df_iv$.
However, it is not hard to see that in any possible location of $y'$, one can always find a 3-feasible cycle $C(i,j,\ell)$ in $\{uv\}\cup f_i\cup f_j\cup xf_\ell y$, a contradiction.
For example, when $y'\in uf_ia\bh\{a\}$, if $E(P_i)\cap af_ix'\neq \emptyset$, then $y'f_\ell x'\cup x'f_ia\cup af_jv\cup \{vu\}\cup uf_iy'$ forms a 3-feasible cycle for $\{f_i,f_j,f_\ell\}$;
otherwise $E(P_i)\cap x'f_ib\neq \emptyset$, then $y'f_\ell x'\cup x'f_ib\cup bf_jv\cup \{vu\}\cup uf_iy'$ forms a 3-feasible cycle for $\{f_i,f_j,f_\ell\}$.

Next, suppose that both $x,y$ lie in one of the paths $af_ib$ and $cf_jd$.
We first assume that $x,y$ lie in $af_ib$ (see Figure 2-c).
Then $f_i\dif f_\ell$ contains the cycle $xf_\ell y\cup xf_iy$.
If $\{f_i, f_\ell\}$ is type-\1, then by Proposition~\ref{prop:2cycleI}, we see $E(xf_iy)\cap E(P_i)=\emptyset$,
which implies that $P_i\subseteq af_ix$ or $P_i\subseteq xf_ib$.
In this case one can easily find a 3-feasible cycle for $\{f_i,f_j,f_\ell\}$, a contradiction.
Therefore, $\{f_i, f_\ell\}$ is not type-\1 and $f_i\dif f_\ell$ consists of the cycle $xf_\ell y\cup xf_iy$.
This shows that $f_\ell=(f_i\bh xf_iy)\cup xf_\ell y \in \dM_{ij}$.
By symmetry, if $x,y$ lie in $cf_jd$, then one can show that $f_\ell\in \dN_{ij}$.

It remains to consider that both $x,y$ lie in $f_k$. Recall that $x\<y$.
We discuss according to the location of $x$.
Assume that $x\<a$. Then we have $b\pq y\pq c$ (as otherwise, one can always find a 3-feasible cycle in $\{uv\} \cup f_i\cup f_j\cup xf_\ell y$ for $\{f_i, f_j, f_\ell\}$, a contradiction).
In this case, since $xf_\ell y$ contains some edge of $P_\ell$, we see that $uf_\ell x=uLx=uf_ix$ and $yf_\ell v=yRv=yf_iv$, implying that $f_\ell\in \dM_{ij}$.

Next assume that $a\pq x\< b$.
We can derive that if $x=a$ then $a\<y\pq c$, and if $a\<x\< b$ then $a\<x\<y\pq b$
(as otherwise there is a 3-feasible cycle for $\{f_i, f_j, f_\ell\}$).
In both cases, we see that the first endpoint of $P_\ell$ precedes $c$ in the linear order $\<$.
Suppose $(f_i,f_\ell)$ is type-\1. Since $c$ is a vertex in $f_i$ with $P_i\<c$, by Proposition~\ref{prop:2cycleI}, there exists a $(u,c)$-path containing $P_i$ and $P_\ell$, whose vertices increase in $\<$ as traversing from $u$ to $c$.
Then this path can be easily extended to a cycle containing $cf_jd\supseteq P_j$, which is 3-feasible for $\{f_i, f_j, f_\ell\}$, a contradiction.
Hence $(f_i,f_\ell)$ is not type-\1.
Suppose that $(f_k,f_\ell)$ is type-\1.
Then $f_\ell\bh f_k$ consists of two paths $xf_\ell y$ and $x'f_\ell y'$.
Since $b,c$ are inner vertices of $P_k$, Proposition~\ref{prop:2cycleI} shows that $y\pq x'\< b\pq c\< y'$, where $x'$ lies in $yf_kb\backslash\{b\}$ and $y'$ lies in $cf_kv\backslash\{c\}$, and $xf_\ell y'$ is internally disjoint from $af_iy'$.
If $x'f_\ell y'$ is internally disjoint from $cf_jd$,
then we can find 3-feasible cycle $af_jx\cup xf_\ell y'\cup y'f_id\cup df_jc\cup cf_ia$ for $\{f_i, f_j, f_\ell\}$, a contradiction.
So $x'f_\ell y'$ intersects with $cf_jd$.
Let $x,y,x',z$ be all splitting vertices in $f_j\cup f_\ell$, where $z\in V(cf_jd)\bh \{c,d\}$.
Then we have $d=y'$ and $x'f_\ell y'=x'f_\ell z\cup zf_jd$.
This shows that $(f_j,f_\ell)$ is type-\1 and by Proposition~\ref{prop:2cycleI}, $P_j\subseteq cf_jz$.
Then $af_ic\cup cf_jz\cup zf_\ell a$ is a 3-feasible cycle for $\{f_i, f_j, f_\ell\}$, a contradiction.
Summarizing, we have that both $(f_i,f_\ell)$ and $(f_k,f_\ell)$ are not type-\1.
Then in either case of $a\pq x\<y\pq b$ and $x=a\<b\<y\pq c$,
we see that $xf_\ell y$ is the unique path in $f_\ell\bh (f_i\cup f_j)$ and $cf_\ell v=cf_iv$, implying that $f_\ell\in \dM_{ij}$.

Putting the above together, we infer $b\pq x$.
By the symmetry between $x$ and $y$, one can also infer that $y\pq c$.
That is $b\pq x\<y\pq c$. Then $uf_ix\cup xf_\ell y\cup yf_jv\cup vu$ forms a 3-feasible cycle for $\{f_i, f_j, f_\ell\}$.
This final contradiction completes the proof of Proposition~\ref{prop:non-feasible-triple-typeI}.
\end{proof}

Recall the definition of $\dW_{ij}$ in the end of Subsection~\ref{subsec:feasible-cycles}.

\begin{prop}\label{prop:MN-typeI}
Let $f_p\in \dM_{ij}\bh \dW_{ij}$ and $f_q\in \dN_{ij}\bh \dW_{ij}$.
Then the paths $x_pf_py_p$ and $x_qf_{q}y_q$ are internally disjoint, where $x_p\<y_p\pq x_q\<y_q$.
\end{prop}
\begin{proof}
Suppose for a contradiction that $x_pf_py_p$ and $x_qf_{q}y_q$ share a common inner vertex $z$.
If $y_p\pq x_q$, then we have $x_p\< z \< y_p\pq x_q\< z\< y_q$, a contradiction.
So we have $x_q\<y_p$, which implies that $x_p\pq a\< b\pq x_q\<z\<y_p\pq c\< d\pq y_q$.
We see $\{f_p,f_q\}$ is type-\1.
By Proposition~\ref{prop:2cycleI},
one of $x_pf_pz\cup zf_{q}y_q$ and $x_qf_{q}z\cup  zf_py_p$ contains some edge in $P_p$ and in $P_q$, respectively.
In either case, we can find a 4-feasible cycle for $\{f_i,f_j,f_p,f_q\}$ and thus $\{f_i,f_j,f_p\}$ is contained in a feasible quadruple,
which shows that $f_p\in \dW_{ij}$, a contradiction.
Hence the paths $x_pf_py_p$ and $x_qf_{q}y_q$ are internally disjoint.
Suppose that $x_q\<y_p$. Then we have $x_p\pq a\< b\pq x_q\<y_p\pq c\< d\pq y_q$.
In this case, one can derive the same contradiction by finding a 4-feasible cycle for $\{f_i,f_j,f_p,f_q\}$.
This shows that $y_p\pq x_q$.
\end{proof}

\begin{prop}\label{prop:typeI-iterate}
Let $\dS$ be any subset of $\dF$ with separator $\{x,y\}$ such that $|\dS|\geq s-\frac{30\sqrt{n}}{\log n}$ and $x\<y$.
Assume that there do not exist two disjoint subsets $\dS_1$ and $\dS_2$ of $\dS$ such that $|\dS_1|+|\dS_2|\geq |\dS|-\frac{52\sqrt{n}}{\log^2 n}$,
$G(\dS_1)$ and $G(\dS_2)$ are edge-disjoint, and each $\dS_i$ contains at most $\sqrt{n}\log^2 n$ type-\1 pairs for $i\in [2]$.\footnote{Here it is possible that one $\dS_i$ is an empty set (if so, $G(\dS_i)$ is an empty graph).}

Then there exists $\dS'\subseteq \dS$ with separator $\{x',y'\}$ such that $|\dS'|\geq |\dS|-\frac{53\sqrt{n}}{\log^2 n}$, $x\pq x'\<y'\pq y$, and
each $(x',y')$-path in $G(\dS')$ can be extended to two distinct $(x,y)$-paths in $G(\dS)$.
\end{prop}

\begin{proof}
If every type-\1 pair $\{f_i,f_j\}$ in $\dS$ has $|\dW_{ij}|\geq 51\sqrt{n}/\log^2 n$,
then by \eqref{equ:Wij}, one can infer that $\dS$ itself contains at most $\sqrt{n}\log^2 n$ type-\1 pairs,
a contradiction to the assumption (as we can just take $\dS_1=\dS$ and $\dS_2=\emptyset$).
So we may assume that there is a type-\1 pair $\{f_i,f_j\}$ in $\dS$ with $|\dW_{ij}|<51\sqrt{n}/\log^2 n$, where $i<j$.
Let $f_k$ be the base of $\{f_i,f_j\}$ and $a\prec b\pq c\prec d$ be all splitting vertices in $f_i\cup f_j$, where $x\pq a$ and $d\pq y$.

By Proposition~\ref{prop:non-feasible-triple-typeI}, any $f_\ell\in \dS\bh (\dW_{ij}\cup\{f_i,f_j,f_k\})$ belongs to $\dM_{ij}\cup \dN_{ij}$.
Let $\dS_1=(\dS\cap \dM_{ij})\bh(\dW_{ij}\cup\{f_i,f_j,f_k\})$ and $\dS_2=(\dS\cap \dN_{ij})\bh(\dW_{ij}\cup\{f_i,f_j,f_k\})$.
Then we have $$|\dS_1|+|\dS_2|=|\dS\bh (\dW_{ij}\cup\{f_i,f_j,f_k\})|\geq |\dS|-52\sqrt{n}/\log^2 n.$$
By Proposition~\ref{prop:MN-typeI}, there exists some vertex $z$ in $bf_ic$ such that any $f_p\in \dS_1$ and $f_q\in \dS_2$ satisfy $V(x_pf_py_p)\pq z\pq V(x_qf_qy_q)$. So $G(\dS_1)$ and $G(\dS_2)$ are edge-disjoint.
Now we claim that for any $\ell\in \{1,2\}$,
any type-\1 pair $\{f_\alpha, f_\beta\}$ in $\dS_\ell$ together with any path $f_\gamma$ in $\dS_{3-\ell}$ form a feasible triple.
Without loss of generality, let $\ell=1$.
Then every vertex in the two cycles of $f_\alpha \dif f_\beta$ precedes $z$ in $\<$,
so $f_\gamma\notin \dM_{\alpha\beta}\cup \dN_{\alpha\beta}$.
Then Proposition \ref{prop:non-feasible-triple-typeI} shows that $\{f_\alpha, f_\beta, f_\gamma\}$ is feasible, proving the claim.

Suppose that every $\dS_\ell$ for $\ell\in [2]$ contains at least $\sqrt{n}/\log^2 n$ paths.
If for some $t\in [2]$, $\dS_t$ contains more than $\sqrt{n}\log^2 n$ type-\1 pairs,
then the above claim shows that there are more than $n$ feasible triples, a contradiction to Proposition~\ref{prop:feasible}.
So we may assume that each of $\dS_1$ and $\dS_2$ contains at most $\sqrt{n}\log^2 n$ type-\1 pairs,
but then such $\dS_1$ and $\dS_2$ contradict our assumption.
Therefore, there exists some $\ell\in [2]$ such that $|\dS_\ell|<\sqrt{n}/\log^2 n$ and thus $|\dS_{3-\ell}|\geq |\dS|-53\sqrt{n}/\log^2 n$.
We now explain that such $\dS_{3-\ell}$ is the desired $\dS'$.
Without loss of generality, let $\ell=1$ and let the separator of $\dS_2$ be $\{x',y'\}$.
Then we can derive $x\pq a\<b\pq z\pq x'\<y'\pq y$.
It is clear that $af_ib$ and $af_jb$ can be extended to two paths in $G(\dS)$ from $x$ to $x'$ and internally disjoint from $G(\dS_2)$.
This completes the proof.
\end{proof}

We are ready to prove Lemma \ref{lem: type 1}.

\begin{proof}[\bf Proof of Lemma \ref{lem: type 1}.]
Assume on the contrary that there do not exist two disjoint subsets $\dS_1,\dS_2$ in $\dF$ satisfying that $|\dS_1|+|\dS_2|>s-30\sqrt{n}/\log n$, $G(\dS_1)$ and $G(\dS_2)$ are edge-disjoint, and each $\dS_i$ contains at most $\sqrt{n}\log^2n$ type-\1 pairs.

Let $\dF_0=\dF$ have separator $\{x_0,y_0\}$ with $x_0\<y_0$. So $|\dF_0|=s+1$.
Suppose that we have defined $\dF_i\subseteq \dF$ for some $i\geq 0$ with separator $\{x_i, y_i\}$ such that $x_i\<y_i$ and $|\dF_i|\geq s-\frac{30\sqrt{n}}{\log n}+\frac{52\sqrt{n}}{\log^2 n}$.
One can easily derive from our assumption that such $\dF_i$ satisfies the conditions of Proposition \ref{prop:typeI-iterate} (for the $\dS$ therein).
So by Proposition \ref{prop:typeI-iterate}, there exists $\dF_{i+1}\subseteq \dF_i$ with separator $\{x_{i+1},y_{i+1}\}$ such that $|\dF_{i+1}|\geq |\dF_i|-\frac{53\sqrt{n}}{\log^2 n}$, $x_i\pq x_{i+1}\<y_{i+1}\pq y_i$, and each $(x_{i+1},y_{i+1})$-path in $G(\dF_{i+1})$ can be extended to two different $(x_i,y_i)$-paths in $G(\dF_i)$.
We repeat this process until the first subset (say $\dF_t$) of size less than 
$s-\frac{30\sqrt{n}}{\log n}+\frac{52\sqrt{n}}{\log^2 n}$  appears.
Since $|\dF_i|-|\dF_{i+1}|\leq \frac{53\sqrt{n}}{\log^2 n}$ for  $0\leq i< t$ and $n$ is sufficiently large,
we have $t\geq \frac{30}{53}\log n-1$.

Note that for each $0\leq i<t$, every $(x_{i+1},y_{i+1})$-path in $G(\dF_{i+1})$ can be extended to two distinct $(x_i,y_i)$-paths in $G(\dF_i)$.
Also $G(\dF_t)$ contains at least $|\dF_t|$ many $(x_t,y_t)$-paths,
where $|\dF_t|\geq s-\frac{30\sqrt{n}}{\log n}-\frac{\sqrt{n}}{\log^2 n}\geq \sqrt n /2$.
Hence, there exist at least $2^t |\dF_t|\geq 2^t \sqrt n/2\geq n$ distinct $(x_0,y_0)$-paths in $G(\dF)=G(\dF_0)$,
which can be extended to at least $n$ distinct $(u,v)$-paths and thus at least $n$ cycles in $G$.
This contradicts \eqref{equ:(n-2)cycles} and completes the proof of Lemma \ref{lem: type 1}.
\end{proof}

\subsection{Proof of Lemma \ref{lem: type 2}}
In this subsection we prove Lemma \ref{lem: type 2}.
Throughout we assume $i<j$ and $\{f_i,f_j\}$ is a type-\2 pair with base $f_k$.
Let $a\prec b\prec c\prec d$ be vertices in $f_k$ such that $\ps(i,k)=af_ic$ and $\ps(j,k)=bf_jd$.

Before we show some structural properties on type-\2 pairs, we point out that $af_ic$ and $bf_jd$ are disjoint.
This is because otherwise there exists some $z\in V(af_ic)\cap V(bf_jd)$ and thus $f_i\dif f_j$ contains at least two cycles,
from which Proposition \ref{prop:diff1} infers that $\{f_i,f_j\}$ is type-\1, a contradiction.

\begin{prop}\label{prop:type-II-struc}
$f_i=uf_ka\cup af_ic\cup cf_kv$, $f_j=uf_kb\cup bf_jd\cup df_kv$, and $E(P_k)\cap E(bf_kc)\neq \emptyset$.
\end{prop}
\begin{proof}
We first show that both $\{f_i,f_k\}$ and $\{f_j,f_k\}$ are not type-\1.

Suppose for a contradiction that $\{f_i,f_k\}$ is a type-\1 pair.
By Proposition~\ref{prop:2cycleI}, $\{f_i,f_k\}$ has a base path $f_\ell$ for some $\ell<\min\{i,k\}$.
Our plan is to show that $f_\ell$ is a crossing path of $\{f_i,f_j\}$, which together with $\ell <k$ would contradict that $f_k$ is the base of $\{f_i,f_j\}$.
To see this, first note that $af_ic\cup af_kc$ is a cycle in $f_i\dif f_k$. So $af_kc=af_\ell c$ and $af_ic=\ps(i,\ell)$.
Also $b$ is a splitting vertex of $f_j\cup f_\ell$ and obviously there is no other splitting vertex of $f_j\cup f_\ell$ in $bf_\ell c$.
By Proposition~\ref{prop:2cycleI} we also have $cf_\ell v=cf_iv$.
Since $\{f_i,f_j\}$ is not type-\1,
there is at most one splitting vertex of $f_i\cup f_j$ in $cf_iv$;
on the other hand, as $c \notin f_j$ there should be some splitting vertex of $f_i\cup f_j$ in $cf_iv$.
Let $d'$ be the unique splitting vertex of $f_i\cup f_j$ in $cf_iv$.
Since $cf_\ell v=cf_iv$, $d'$ is also a splitting vertex of $f_j\cup f_\ell$.
If $d\in V(bf_jd')$, as $E(P_j)\cap E(bf_jd)\neq \emptyset$, we have $\ps(j,\ell)=bf_jd'$.
Since $a\< b\< c\< d'$ lie in $f_\ell$ with $af_ic=\ps(i,\ell)$ and $\ps(j,\ell)=bf_jd'$,
this shows that $f_\ell$ is a crossing path of $\{f_i,f_j\}$, a contradiction.
Hence, we have $d'\in V(bf_jd)$. If $E(P_j)\cap E(bf_jd')\neq \emptyset$, then again we have $\ps(j,\ell)=bf_jd'$, which also shows that $f_\ell$ is a crossing path of $\{f_i,f_j\}$.
So we may assume that $E(P_j)\cap E(bf_jd')=\emptyset$ and thus $P_j\subseteq d'f_jv$.
However in this case, since $\{f_i.f_j\}$ is type-\2,
the cycle $f_i\dif f_j$ belongs to $uf_id'\cup uf_jd'$, which cannot contain any edge in $P_j$, a contradiction to Proposition \ref{prop:diff1}.
This proves that $\{f_i,f_k\}$ is not type-\1 and thus $f_i=uf_ka\cup af_ic\cup cf_kv$.
Similarly, we can show $f_j=uf_kb\cup bf_jd\cup df_kv$.

By Proposition \ref{prop:P-notin-f}, each of $\ps(k,i)=af_kc$ and $\ps(k,j)=bf_kd$ contains some edge in $P_k$.
Therefore, we see that $bf_kc$ contains some edge in $P_k$, completing the proof.
\end{proof}

Let $\dA_{ij}$ consist of all paths $f_\ell\in \dF$ with $\ell\notin \{i,j,k\}$ satisfying that $f_\ell=uf_ix_\ell \cup x_\ell f_\ell y_\ell\cup y_\ell f_iv$,
where $x_\ell, y_\ell\in V(af_ic)$ are splitting vertices in $f_\ell\cup f_i$ and $x_\ell f_\ell y_\ell$ is disjoint from $f_k\cup f_j$.
Let $\dB_{ij}$ consist of all paths $f_\ell\in \dF$ with $\ell\notin \{i,j,k\}$ satisfying that $f_\ell=uf_jx_\ell \cup x_\ell f_\ell y_\ell\cup y_\ell f_jv$,
where $x_\ell, y_\ell\in V(bf_jd)$ are splitting vertices in $f_\ell\cup f_j$ and $x_\ell f_\ell y_\ell$ is disjoint from $f_k\cup f_i$.


\begin{prop}\label{prop:non-feasible-triple-type2}
Let $\ell\notin \{i,j,k\}$. If $f_\ell \notin \dW_{ij}$, then $f_\ell \in \dA_{ij}\cup \dB_{ij}$.
\end{prop}
\begin{proof}
Fix some $\ell\notin \{i,j,k\}$ with $f_\ell \notin \dW_{ij}$.
Our general proof strategy is, using the symmetry between $x$ and $y$ and the symmetry between $af_ic$ and $bf_jd$,
either to show $f_\ell\in \dA_{ij}\cup \dB_{ij}$, or to find a 3-feasible cycle $C(i,j,\ell)$, or to find a 4-feasible cycle $C(i,j,k,\ell)$.
Each of the latter two cases implies that $f_\ell \in \dW_{ij}$ and thus reaches a contradiction.

We first show that there exists some path say $xf_\ell y$ in $f_\ell\backslash (f_i\cup f_j\cup f_k)$, which contains some edge in $P_\ell$.
Otherwise, $P_\ell \subseteq f_i\cup f_j\cup f_k$, then by Proposition~\ref{prop:degenerate-3}
$f_\ell$ is the base of some type-\1 pair in $\{f_i, f_j, f_k\}$;
however, by Proposition~\ref{prop:type-II-struc} none of the pair in $\{f_i, f_j, f_k\}$ is type-\1, a contradiction.

We shall divide the coming proof into several cases by considering the locations of $x,y$.
Note that by Proposition~\ref{prop:type-II-struc}, each of $x,y$ lies in $af_ic, bf_jd$ or $f_k$.
Before we proceed, it will be convenient for later use to collect some properties about $af_ic, bf_jd, bf_kc$ and $xf_\ell y$.
Given $\alpha\in I\subseteq \{0\}\cup [s]$, we write $\cup I\bh \alpha$ for the union of paths $f_\beta$ for all $\beta\in I\bh\{\alpha\}$.
A path $Q$ is called {\it $\alpha$-unique} for $I$, if either $P_\alpha\bh (\cup I\bh \alpha)=\emptyset$ or $E(Q)\cap E(P_\alpha\bh (\cup I\bh \alpha))\neq \emptyset$.
We claim that
\begin{align}\label{equ:ijkl-uniqueness}
af_ic \mbox{ is } i\mbox{-unique, } bf_jd \mbox{ is } j\mbox{-unique, } bf_kc \mbox{ is } k\mbox{-unique, and } xf_\ell y \mbox{ is } \ell\mbox{-unique for } \{i,j,k,\ell\}.
\end{align}
This is clear for $xf_\ell y$.
Considering $bf_kc$, we have either $E(bf_kc)\cap E(P_k\bh (f_i\cup f_j\cup f_\ell))\neq \emptyset$ or $E(bf_kc)\cap E(P_k)\subseteq E(f_\ell)$,
the latter of which implies that $P_k\subseteq f_i\cup f_j\cup f_\ell$ (using the structure of Proposition~\ref{prop:type-II-struc}).
Hence, $bf_kc$ is $k$-unique for $\{i,j,k,\ell\}$.
Similarly, one can see that $af_ic$ is $i$-unique and $bf_jd$ is $j$-unique for $\{i,j,k,\ell\}$, completing the proof of \eqref{equ:ijkl-uniqueness}.
Also we claim that
\begin{align}\label{equ:ijl-uniqueness}
af_ic \mbox{ is } i\mbox{-unique, } bf_jd \mbox{ is } j\mbox{-unique, and } xf_\ell y \mbox{ is } \ell\mbox{-unique for } \{i,j,\ell\}.
\end{align}
This is also clear for $xf_\ell y$. Now suppose $E(af_ic)\cap E(P_i)\subseteq E(f_\ell)$, which implies $P_i\subseteq f_k\cup f_\ell$.
By Proposition~\ref{prop:degenerate-2}, $\{f_k,f_\ell\}$ is type-\1 with base $f_i$ and thus $f_i\subseteq f_k\cup f_\ell$.
Note that the splitting vertices $a,c$ in $f_i\cup f_k$ are also splitting vertices in $f_k\cup f_\ell$.
Since $b,d$ are the splitting vertices in $f_j\cup f_k$ such that $a\<b\<c\<d$ lie in $f_k$,
we can conclude that $f_j \not \in \dM_{k\ell}\cup \dN_{k\ell}$ and
thus Proposition \ref{prop:non-feasible-triple-typeI} implies that $\{f_j,f_k,f_\ell\}$ is feasible.
Using the fact that $f_i\subseteq f_k\cup f_\ell$, one can verify that any $3$-feasible cycle for $\{f_j,f_k,f_\ell\}$ is a 4-feasible cycle for $\{f_i,f_j,f_k,f_\ell\}$.
This shows that $f_\ell \in \dW_{ij}$, a contradiction.
Hence $E(af_ic)\cap E(P_i)\not\subseteq E(f_\ell)$, which implies that $af_ic$ is $i$-unique for $\{i,j,\ell\}$.
Analogously, $E(bf_jd)\cap E(P_j)\not\subseteq E(f_\ell)$ and thus $bf_jd$ is $j$-unique for $\{i,j,\ell\}$, proving \eqref{equ:ijl-uniqueness}.

\begin{figure}[ht!]\label{fig:A}
\begin{minipage}{0.33 \linewidth}
        \centering
\begin{tikzpicture}[scale=0.7,line cap=round,line join=round,>=triangle 45,x=1cm,y=1cm]
\clip(-0.75,-0.5) rectangle (7.75,2.5);
\draw [line width=1pt] (0,0)-- (7,0);

\draw [line width=1pt,color=red] (1,0 ) to[out=90,in=90]  (4.5,0);

\draw [line width=1pt,color=blue] (2.5,0 ) to[out=90,in=90]  (6,0);


 \draw [line width=0.5pt] (1.5,0.74)  .. controls (1.75,2) and  ( 2.5,2) ..   (3,1.03);
\begin{scriptsize}
\draw [fill=black] (0,0) circle (1.5pt);
\draw[color=black] (0 ,-0.25) node {$u$};
\draw [fill=black] (7,0) circle (1.5 pt);
\draw[color=black] (7 ,-0.25) node {$v$};

\draw [fill=black] (1,0) circle (1.5pt);
\draw[color=black] (1 ,-0.25) node {$a$};
\draw [fill=black] (2.5,0) circle (1.5pt);
\draw[color=black] (2.5 ,-0.25) node {$b$};
\draw [fill=black] (4.5,0) circle (1.5pt);
\draw[color=black] (4.5,-0.25) node {$c$};
\draw [fill=black] (6,0) circle (1.5pt);
\draw[color=black] (6 ,-0.25) node {$d$};

\draw [fill=black] (1.5,0.74) circle (1.5pt);
\draw[color=black] (1.5,0.5) node {$x$};

\draw [fill=black] (3,1.03) circle (1.5pt);
\draw[color=black] (2.7,0.8) node {$y$};

\draw[color=black] (1,0.8) node {$f_i$};
\draw[color=black] (6.5, 0.3) node {$f_k$};
\draw[color=black] (4.5 ,1.3) node {$f_j$};
\draw[color=black] (3 ,1.6) node {$f_{\ell}$};
\end{scriptsize}
\end{tikzpicture}
\caption*{(a)}
\end{minipage}
\begin{minipage}{0.33\linewidth}
        \centering
\begin{tikzpicture}[scale=0.7,line cap=round,line join=round,>=triangle 45,x=1cm,y=1cm]
\clip(-0.75,-0.5) rectangle (7.75,2.5);
\draw [line width=1pt] (0,0)-- (7,0);
\draw [line width=1pt,color=red] (1,0 ) to[out=90,in=90]  (4.5,0);
\draw [line width=1pt,color=blue] (2.5,0 ) to[out=90,in=90]  (6,0);

\draw [line width=0.5pt] (2.2,1) to[out=50,in=130]  (4.7,1.03);

\begin{scriptsize}
\draw [fill=black] (0,0) circle (1.5pt);
\draw[color=black] (0 ,-0.25) node {$u$};
\draw [fill=black] (7,0) circle (1.5 pt);
\draw[color=black] (7 ,-0.25) node {$v$};

\draw [fill=black] (1,0) circle (1.5pt);
\draw[color=black] (1 ,-0.25) node {$a$};
\draw [fill=black] (2.5,0) circle (1.5pt);
\draw[color=black] (2.5 ,-0.25) node {$b$};
\draw [fill=black] (4.5,0) circle (1.5pt);
\draw[color=black] (4.5,-0.25) node {$c$};
\draw [fill=black] (6,0) circle (1.5pt);
\draw[color=black] (6 ,-0.25) node {$d$};

\draw [fill=black] (2.2,1) circle (1.5pt);
\draw[color=black] (2.2,0.8) node {$x$};

\draw [fill=black] (4.7,1) circle (1.5pt);
\draw[color=black] (4.7,0.8) node {$y$};

\draw[color=black] (1,0.8) node {$f_i$};
\draw[color=black] (6 ,0.8) node {$f_j$};
\draw[color=black] (2.5 ,1.6) node {$f_{\ell}$};
\draw[color=black] (6.5, 0.3) node {$f_k$};
\end{scriptsize}
\end{tikzpicture}
\caption*{(b)}
\end{minipage}
\begin{minipage}{0.33\linewidth}
        \centering
\begin{tikzpicture}[scale=0.7,line cap=round,line join=round,>=triangle 45,x=1cm,y=1cm]
\clip(-0.75,-0.5) rectangle (7.75,2.5);
\draw [line width=1pt] (0,0)-- (7,0);
\draw [line width=1pt,color=red] (1,0 ) to[out=90,in=90]  (4.5,0);
\draw [line width=1pt,color=blue] (2.5,0 ) to[out=90,in=90]  (6,0);

\draw [line width=0.5pt] (3.5,0) to[out=90,in=200]  (4,0.75) ;

\begin{scriptsize}
\draw [fill=black] (0,0) circle (1.5pt);
\draw[color=black] (0 ,-0.25) node {$u$};
\draw [fill=black] (7,0) circle (1.5 pt);
\draw[color=black] (7 ,-0.25) node {$v$};

\draw [fill=black] (1,0) circle (1.5pt);
\draw[color=black] (1 ,-0.25) node {$a$};
\draw [fill=black] (2.5,0) circle (1.5pt);
\draw[color=black] (2.5 ,-0.25) node {$b$};
\draw [fill=black] (4.5,0) circle (1.5pt);
\draw[color=black] (4.5,-0.25) node {$c$};
\draw [fill=black] (6,0) circle (1.5pt);
\draw[color=black] (6 ,-0.25) node {$d$};

\draw [fill=black] (3.5,0) circle (1.5pt);
\draw[color=black] (3.5,-0.25) node {$y$};

\draw [fill=black] (4,0.75) circle (1.5pt);
\draw[color=black] (4.2,0.8) node {$x$};

\draw[color=black] (1,0.8) node {$f_i$};
\draw[color=black] (6.5, 0.3) node {$f_k$};
\draw[color=black] (6 ,0.8) node {$f_j$};
\draw[color=black] (3.35 ,0.5) node {$f_{\ell}$};
\end{scriptsize}
\end{tikzpicture}
\caption*{(c)}
\end{minipage}
\caption*{Figure 3}
\end{figure}
Now consider the case that both of $x,y$ lie in $af_ic$ or $bf_jd$.
By symmetry, let us assume $x,y\in V(af_ic)$ with $x\<y$  (see Figure 3-a).
If $\{f_i, f_\ell\}$ is not type-\1, then it is easy to see that $f_\ell \in \dA_{ij}$.
Now suppose $\{f_i, f_\ell\}$ is type-\1.
By Proposition~\ref{prop:2cycleI}, $P_\ell\subseteq xf_\ell y$ and $E(xf_iy)\cap E(P_i)=\emptyset$.
As $\ps(i,k)=af_ic$, we see that $R=af_ix\cup xf_\ell y\cup yf_ic$ is an $(a,c)$-path,
which is both $i$-unique and $\ell$-unique for $\{i,j,\ell\}$.
So by \eqref{equ:ijl-uniqueness}, $R\cup cf_id\cup df_ja$ forms a 3-feasible cycle $C(i,j,\ell)$.

Suppose that $x\in V(af_ic)\backslash \{a,c\}$ and $y\in V(bf_jd)\backslash \{b,d\}$ (see Figure 3-b).
By \eqref{equ:ijl-uniqueness}, it is easy to see that one of the following four cycles can form a 3-feasible cycle $C(i,j,\ell)$.
These include: (a) $xf_\ell y\cup yf_ja\cup af_ix$, (b) $xf_\ell y\cup yf_jv\cup \{uv\}\cup uf_ix$, (c) $xf_\ell y\cup yf_jd\cup df_ix$, and (d) $xf_\ell y\cup yf_ju\cup \{uv\}\cup vf_ix$.

Next suppose that both $x,y$ lie in $f_k\backslash \{a,d\}$. Let $x\<y$.
If $x,y\notin V(bf_kc)$, by \eqref{equ:ijkl-uniqueness} and \eqref{equ:ijl-uniqueness} one can always find either a 3-feasible cycle $C(i,j,\ell)$ containing $af_ic\cup bf_jd\cup xf_\ell y$
or a 4-feasible cycle $C(i,j,k,\ell)$ containing $af_ic\cup bf_jd\cup bf_kc \cup xf_\ell y$.
So by symmetry between $x$ and $y$, we may assume that $y\in V(bf_kc)$.
If $x\< a$, then by \eqref{equ:ijkl-uniqueness}, one of the cycles $xf_\ell y\cup yf_kb\cup bf_jd\cup df_ix$ and $xf_\ell y\cup yf_kc\cup cf_ia\cup af_jv\cup \{uv\}\cup uf_kx$ gives a 4-feasible cycle $C(i,j,k,\ell)$.
If $a\< x\< b$, then one of the cycles $xf_\ell y\cup yf_kb\cup bf_jd\cup df_ia\cup af_kx$ and $xf_\ell y\cup yf_kc\cup cf_iu\cup \{uv\}\cup vf_jx$ gives a 4-feasible cycle $C(i,j,k,\ell)$.
Then it remains to consider $b\pq x\<y\pq c$.
If $E(P_k)\cap E(xf_ky)\neq\emptyset$,
then $\{f_k,f_\ell\}$ is not type-\1 and thus $f_\ell=uf_kx\cup xf_\ell y\cup yf_kv$,
which implies that $uf_ic\cup cf_\ell b\cup bf_jv\cup \{vu\}$ forms a 3-feasible cycle $C(i,j,\ell)$.
So $E(P_k)\cap E(xf_ky)=\emptyset$.
By \eqref{equ:ijkl-uniqueness}, one of $bf_kx$ and $yf_kc$ is $k$-unique for $\{i,j,k,\ell\}$.
Therefore, $xf_\ell y\cup yf_kc\cup cf_iu\cup \{uv\}\cup vf_jb\cup bf_kx$ forms a 4-feasible $C(i,j,k,\ell)$, proving this case.

The only possible case left is that one of $x,y$ is in $V(af_ic\cup bf_jd)\bh \{b,c\}$ and the other is in $V(f_k)$.
By symmetry, we may assume that $x\in V(af_ic)\bh\{c\} $ and $y\in V(f_k)\bh\{a,c\}$ (see Figure 3-c).
Recall that in the proof of \eqref{equ:ijl-uniqueness}, we proved $E(af_ic)\cap E(P_i)\not\subseteq E(f_\ell)$.
So either (a) $E(af_ix)\cap E(P_i)\not\subseteq E(f_\ell)$, or (b) $E(xf_ic)\cap E(P_i)\not\subseteq E(f_\ell)$.
First suppose (a) occurs.
If $y\in V(uf_kb\cup df_kv)\backslash \{a,c\}$,
then there is a $(y,a)$-path $Q$ in the cycle $\{uv\}\cup f_j$ containing $bf_jd$.
In this case, we see that $Q\cup af_ix\cup xf_\ell y$ forms a 3-feasible cycle $C(i,j,\ell)$.p
So we may assume $y\in V(bf_kd)\backslash \{c\}$.
Then one of the cycles $uf_ix\cup xf_\ell y \cup yf_kb \cup bf_jv \cup \{uv\}$ and $af_ix\cup xf_\ell y \cup yf_kd \cup df_ja$ is a 4-feasible cycle for $\{f_i,f_j,f_k,f_\ell\}$.

Therefore (b) occurs.
If $y\pq b$, then $yf_\ell x\cup xf_id\cup df_jy$ forms a 3-feasible cycle $C(i,j,\ell)$.
If $c\<y$, then $yf_\ell x\cup xf_ic\cup cf_kb\cup bf_jd \cup df_ky$ forms a 4-feasible cycle $C(i,j,k,\ell)$.
So it only remains to consider $y\in V(bf_kc)\backslash\{b,c\}$.
If $bf_ky$ is $k$-unique for $\{i,j,k,\ell\}$, then $bf_ky\cup yf_\ell x\cup xf_id\cup df_jb$ forms a 4-feasible cycle $C(i,j,k,\ell)$, a contradiction.
By \eqref{equ:ijkl-uniqueness} we have $E(yf_kc)\cap E(P_k\bh f_\ell)\neq \emptyset$.
In particular, $yf_kc\not\subseteq f_\ell$.
Also we have $bf_ky\not\subseteq f_\ell$ (as otherwise $bf_\ell x\cup xf_id \cup df_jb$ is a 3-feasible cycle $C(i,j, \ell)$).
Note that $y$ is a splitting vertex in $f_k\cup f_\ell$, so $\{f_k,f_\ell\}$ must be type-\1, with base say $f_{k'}$.
Since $E(yf_kc)\cap E(P_k\bh f_\ell)\neq \emptyset$ and $bf_ky\not\subseteq f_\ell$,
by Proposition~\ref{prop:2cycleI}, $y$ can not be the first or the last splitting vertex in $f_k\cup f_\ell$.
So $x\< y$ (as otherwise, $y$ must be the first splitting vertex in $f_k\cup f_\ell$) and $y$ is an inner vertex of $P_{k'}$.
This further implies that $f_{k'}=uf_ky\cup yf_\ell v$, and $bf_ky$ is $k'$-unique and $xf_\ell y$ is $\ell$-unique for $\{i,j,k',\ell\}$.
Because of $(b)$ and $E(bf_jd)\cap E(P_j)\not\subseteq E(f_\ell)$ (proved in the proof of \eqref{equ:ijl-uniqueness}),
one can derive that $xf_ic$ is $i$-unique and $bf_jd$ is $j$-unique for $\{i,j,k',\ell\}$.
Then $xf_\ell y\cup yf_kb\cup bf_jd\cup df_ix$ forms a 4-feasible cycle for $\{f_i,f_j,f_{k'},f_\ell\}$,
so $f_\ell\in \dW_{ij}$.
This final contradiction finishes the proof of Proposition~\ref{prop:non-feasible-triple-type2}.
\end{proof}

\begin{prop}\label{prop:AB-typeII}
For $f_\alpha\in \dA_{ij}\bh \dW_{ij}$ and $f_\beta\in \dB_{ij}\bh \dW_{ij}$,
the paths $x_\alpha f_\alpha y_\alpha$ and $x_\beta f_\beta y_\beta$ are disjoint.
\end{prop}
\begin{proof}
By Proposition~\ref{prop:type-II-struc}, we see that $a,d$ are two splitting vertices in $f_\alpha \cup f_\beta$.
Suppose for a contradiction that $x_\alpha f_\alpha y_\alpha$ and $x_\beta f_\beta y_\beta$ have a common vertex $z$.
It is clear that $z$ is a splitting vertex in $f_\alpha \cup f_\beta$ with $a\<z\<d$.
So $\{f_\alpha, f_\beta\}$ is type-\1. By Proposition~\ref{prop:2cycleI},
exactly one of $Q_1=x_\alpha f_\alpha z\cup zf_\beta y_\beta$ and $Q_2=x_\beta f_\beta z\cup zf_\alpha y_\alpha$
contains some edges in $P_\alpha\bh (f_i\cup f_j\cup f_\beta)$ and in $P_\beta\bh (f_i\cup f_j\cup f_\alpha)$.
First assume that $Q_1$ does so.
Note that either $af_ix_\alpha$ or $x_\alpha f_ic$ contains some edge in $P_i\bh (f_j\cup f_\alpha\cup f_\beta)$,
and either $bf_jy_\beta$ or $y_\beta f_jd$ contains some edge in $P_j\bh (f_i\cup f_\alpha\cup f_\beta)$.
So there are four possibilities and it is not hard to verify that one can always find a 4-feasible cycle $C(i,j,\alpha,\beta)$ containing $Q_1$ in each possibility.
The proof for the other case (that is, $Q_2$ contains those edges mentioned above) is analogous,
and we can always find a 4-feasible cycle $C(i,j,\alpha,\beta)$ containing $Q_2$.
In any case, we see that both $f_\alpha$ and $f_\beta$ are contained in $\dW_{ij}$, a contradiction.
\end{proof}

By comparing Proposition \ref{prop:AB-typeII} with Proposition \ref{prop:MN-typeI},
we see that $\dA_{ij}, \dB_{ij}$ can play the same roles of $\dM_{ij}, \dN_{ij}$ as in Subsection \ref{Subsec:typeI}.
The next result is analogue to Proposition \ref{prop:typeI-iterate}.

\begin{prop}\label{prop:typeII-iterate}
Let $\dT$ be any subset of $\dF$ with separator $\{x,y\}$ such that $|\dT|\geq s-\frac{60\sqrt{n}}{\log n}$ and $x\<y$.
Assume that there do not exist two disjoint subsets $\dT_1$ and $\dT_2$ of $\dT$ such that $|\dT_1|+|\dT_2|\geq |\dT|-\frac{52\sqrt{n}}{\log^2 n}$,
$G(\dT_1)$ and $G(\dT_2)$ are edge-disjoint, and each $\dT_i$ contains at most $\sqrt{n}\log^2 n$ type-\2 pairs for $i\in [2]$.

Then there exists $\dT'\subseteq \dT$ with separator $\{x',y'\}$ such that $|\dT'|\geq |\dT|-\frac{103\sqrt{n}}{\log^2 n}$, $x\pq x'\<y'\pq y$, and
each $(x',y')$-path in $G(\dT')$ can be extended to two distinct $(x,y)$-paths in $G(\dT)$.
\end{prop}

\begin{proof}
The proof will follow the lines as in Proposition \ref{prop:typeI-iterate}.
So we will omit details when the corresponding argument is the same as before.

First, we may assume that there is a type-\2 pair $\{f_i,f_j\}$ in $\dT$ with $|\dW_{ij}|<51\sqrt{n}/\log^2 n$.
Let $f_k$ be the base of $\{f_i,f_j\}$ and let $a\<b\<c\<d$ be vertices in $f_k$ such that $\ps(i,k)=af_ic$ and $\ps(j,k)=bf_jd$.
Let $\dT_1=(\dT\cap \dA_{ij})\bh(\dW_{ij}\cup\{f_i,f_j,f_k\})$ and $\dT_2=(\dT\cap \dB_{ij})\bh(\dW_{ij}\cup\{f_i,f_j,f_k\})$.
By Proposition~\ref{prop:non-feasible-triple-type2}, $$|\dT_1|+|\dT_2|=|\dT\bh (\dW_{ij}\cup\{f_i,f_j,f_k\})|\geq |\dT|-52\sqrt{n}/\log^2 n.$$
Let $\{x_A,y_A\}$ and $\{x_B, y_B\}$ be the separators of $\dT_1$ and $\dT_2$, respectively.
Then $a\pq x_A\<y_A\pq c$ lie in $af_ic$, $b\pq x_B\<y_B\pq d$ lie in $bf_jd$, and by Proposition \ref{prop:AB-typeII}, $G(\dT_1)$ and $G(\dT_2)$ are disjoint.
We claim that
\begin{align*}
\mbox{for any $\ell\in \{1,2\}$ and any type-\2 pair $\{f_\alpha, f_\beta\}$ in $\dT_\ell$, we have $\dT_{3-\ell}\subseteq \dW_{\alpha\beta}$.}
\end{align*}
To see this, note that $f_\alpha\dif f_\beta\subseteq G(\dT_\ell)$.
So any path in $\dT_{3-\ell}$ does not belong to $\dA_{\alpha\beta}\cup \dB_{\alpha\beta}$.
Then Proposition \ref{prop:non-feasible-triple-type2} shows that $\dT_{3-\ell}\subseteq \dW_{\alpha\beta}$, as claimed.

Suppose $|\dT_\ell|\geq 51\sqrt{n}/\log^2 n$ for every $\ell\in [2]$.
If there exists some $t\in [2]$ such that $\dT_t$ contains more than $\sqrt{n}\log^2 n$ type-\2 pairs,
then the above claim would derive a contradiction to \eqref{equ:Wij};
otherwise, each of $\dT_1$ and $\dT_2$ contains at most $\sqrt{n}\log^2 n$ type-\2 pairs, a contradiction to our assumption.
Therefore, we may assume $|\dT_1|<51\sqrt{n}/\log^2 n$ and thus $|\dT_2|\geq |\dT|-103\sqrt{n}/\log^2 n$.
Recall the separator $\{x_B, y_B\}$ of $\dT_2$ such that $x\pq a\<b\pq x_B\<y_B\pq d\pq y$ lie in $f_j$.
To show that $\dT_2$ is the desired $\dT'$,
it suffices to see that each $(x_B,y_B)$-path $R_0$ in $G(\dT_2)$ can be extended to two $(x,y)$-paths
$R_1\cup R_0\cup y_Bf_jy$ and $R_2\cup R_0\cup y_Bf_jy$ in $G(\dT)$, where $R_1=xf_jx_B$ and $R_2=xf_ja\cup af_ic\cup cf_kb\cup bf_jx_B$.
\end{proof}

We can then prove Lemma \ref{lem: type 2} promptly.

\begin{proof}[\bf Proof of Lemma \ref{lem: type 2}.]
By replacing Proposition \ref{prop:typeI-iterate} with Proposition \ref{prop:typeII-iterate},
this proof is identical to the proof of Lemma \ref{lem: type 1}.
We left the verification to readers.
\end{proof}

Now the proof of the main result Lemma~\ref{lem:normal} in this section  is completed.

\section{Reordering and partitioning $\dF$}\label{sec:reorder-dF}
The goal of this section is to show that roughly speaking, one can reorder most paths in $\dF$ and
partition them into a bounded number of intervals such that for every relevant edge $e$, paths containing $e$ in each interval are listed almost consecutively.
The precise statement is as follows.

\begin{lem}\label{lem: consecutive}
There exist disjoint subsets $\dP_1, \dP_2, \dP_3, \dP_4$ in $\dF$ and constants $\beta, \gamma$ with $\beta\geq n^{1/4}\log n$ and $\gamma=\beta \log n\leq \sqrt{n}$ such that the following hold:
\begin{itemize}
\item [1).] $\sum_{i\in [4]} |\dP_i|\geq (1-o(1))\cdot s$, all $G(\dP_i)$'s are edge-disjoint, and each $\dP_i$ contains at most $2\sqrt{n}\log^2 n$ pairs of type-\1 and type-\2.
\item [2).] For any edge $e$ in $G(\dP_i)$, let $d(e)$ denote the number of paths in $\dP_i$ containing $e$.
Then there are at most $9n\log \log n/\log n$ edges $e$ in $G(\dP_i)$ satisfying that $\beta\leq d(e)\leq \gamma$.
\item [3).] Each $\dP_i$ has an arrangement $\{g_j\}_{j\geq 1}$ and a partition of at most $3\sqrt{n}/\gamma$ intervals
such that the following holds.\footnote{Here, an {\it interval} of $\dP_i$ means a subset of $\dP_i$ consisting of paths $g_j$ for all integers $j$ in some interval $[a,b]$.}
For any edge $e$ in $G(\dP_i)$ with $d(e)\geq \gamma$, one can delete at most $3\beta$ paths in $\dP_i$ such that
there is at most one interval of $\dP_i$ which can contain some remaining paths $g_j, g_k$ with $e\in E(g_j)$ and $e\notin E(g_k)$ and moreover, all such paths $g_j, g_k$ satisfy $j<k$.
\end{itemize}
\end{lem}

We devote the rest of this section to the proof of Lemma \ref{lem: consecutive}.
We begin by defining the desired subsets $\dP_i$ of $\dF$.
Let $\dF_1,\dF_2,\dF_3,\dF_4$ be the four disjoint subsets of $\dF$ from Lemma~\ref{lem:normal}.

\begin{dfn}
For each $i\in [4]$, let $\dP_i$ be obtained from $\dF_i$ by deleting all paths each of which is contained in
at least $26\sqrt{n}\log n$ sets $\dW_{jk}$'s
or at least $n^{1/4}$ pairs of type-\1 and type-\2 in $\dF_i$.
\end{dfn}
From Lemma~\ref{lem:normal} we see that the number of type-\1 and type-\2 pairs in each $\dF_i$ is at most 2$\sqrt{n}\log^2 n$.
Together with \eqref{equ:Wij}, we have $|\dP_i|\geq |\dF_i|-n^{1/2}/(2\log n)-2n^{1/4}\log^2 n$.
So by Lemma~\ref{lem:normal} again, we can derive that
$$|\cup_{i\in [4]} \dP_i|\geq |\cup_{i\in [4]} \dF_i|-2\sqrt{n}/\log n-8n^{1/4}\log^2 n=(1-o(1))\cdot s.$$
Since $\dP_i\subseteq \dF_i$ implies $G(\dP_i)\subseteq G(\dF_i)$,
we see that $G(\dP_i)$'s are pairwise edge-disjoint and each $\dP_i$ contains at most $2\sqrt{n}\log^2 n$ pairs of type-\1 and type-\2.
This proves the first item of Lemma \ref{lem: consecutive}.

Next we define the constants $\beta$ and $\gamma$.
For each edge $e$ contained in $G(\dP_i)$ for some $i\in [4]$, we define its {\it degree} $d(e)$ to be the number of paths in $\dP_i$ containing $e$.
So obviously $1\le d(e)\leq s+1\leq 2\sqrt{n}$ for every such $e$.
Let $A=\big\lceil\frac{\log n}{4\log \log n}\big\rceil-2$ and let $\alpha_0,\alpha_1,..., \alpha_A$ be a geometric sequence of reals such that
$\alpha_j=n^{1/4}(\log n)^{j+1}$.
By average, there is some $j_0\in [A]$ such that
the number of edges $e$ with $d(e)\in [\alpha_{j_0-1},\alpha_{j_0}]$ is at most $\frac{2e(G)}{A}\leq \frac{9n\log \log n}{\log n}.$
Let $\beta=\alpha_{j_0-1}$ and $\gamma=\alpha_{j_0}$.
So $\beta\geq n^{1/4}\log n$ and $\gamma=\beta \log n\leq \sqrt{n}$, as wanted.


It remains to show the third item of Lemma \ref{lem: consecutive}.
For this, in the rest of this section we shall focus on one of $\dP_{i}$'s and express it as $\dP$.

We need to collect some properties on the edges in $\dP$ first.
Let $\{u_0,v_0\}$ be the separator of $\dP$ with $u_0\<v_0$.
Recall the spanning trees $L$ and $R$ from \eqref{equ:LR}.
So all paths in $\dP$ contain $uLu_0\cup v_0Rv$.

\begin{prop}\label{prop:d(e)n1/4}
Let $e\in E(G(\dP))$. If there are $f_k, f_\ell \in \dP$ with $e\in E(L_k)\cap E(R_\ell)$,
then $d(e)\leq 2n^{\frac14}$.
\end{prop}
\begin{proof}
Suppose that $d(e)>2n^{1/4}$.
Then there are at least $2n^{1/4}$ paths $f_j$ in $\dP\bh\{f_k,f_\ell\}$ with $e\in E(L_j)\cup E(R_j)$.
If $e\in E(R_j)$, then $e\in E(R_j)\cap E(L_k)$ and by Proposition~\ref{prop:P-notin-f}, $\{f_j,f_k\}$ must be type-\1.
Similarly, if $e\in E(L_j)$, then we also can see that $\{f_j,f_\ell\}$ is type-\1.
Thus, one of $f_k$ and $f_\ell$ is contained in at least $n^{1/4}$ type-\1 pairs in $\dP$.
However this is a contradiction to the definition of $\dP$.
\end{proof}

\begin{prop}\label{prop:d(e)>gamma}
All edges $e\in E(G(\dP))$ with $d(e)\geq \gamma$ induce two edge-disjoint
trees $L_\dP$ and $R_\dP$ with roots $u_0$ and $v_0$, which are subtrees of $L$ and $R$, respectively.
\end{prop}
\begin{proof}
Consider such edges $e$ with $d(e)\geq \gamma\geq \alpha_0>2n^{1/4}$.
By Proposition~\ref{prop:d(e)n1/4}, there are two kinds of such edges $e$:
either (1) all paths $f_j$ in $\dP$ containing $e$ satisfy $e\in E(L_j\cup P_j)$,
or (2) all paths $f_j$ in $\dP$ containing $e$ satisfy $e\in E(P_j\cup R_j)$.
Let $e=xy$ with $x\<y$.
In the former case (1), we see that all paths in $\dP$ containing $e$ also contain the path $L[u_0,y]$,
so all edges in $L[u_0,y]$ have degree at least $d(e)\geq \gamma$.
This shows that all edges satisfying (1) induce a subtree of $L$ (with root $u_0$).
The analog also holds for the latter case. This finishes the proof.
\end{proof}

\begin{dfn}
Each leaf-edge in the rooted tree $L_\dP$ or $R_\dP$ is called a {\it transforming edge} of $\dP$.

For any paths $f,g\in \dP$, let the first and the last splitting vertices in $f\cup g$ (according to the linear ordering $\<$)
be $\ls(f,g)$ and $\rs(f,g)$, respectively.
\end{dfn}

We now define an arrangement $\{g_j\}_{j\geq 1}$ for $\dP$ in the following algorithm.

\medskip

{\noindent \bf Algorithm for ordering the paths in $\dP$.}
Initially, set $j=1$ and $\dS=\dP$. We iterate the following three steps until $\dS=\emptyset$.
\begin{itemize}
\item[(a).] If $j$=1, let $x=u_0, y=v_0$ and $\dS_1=\dS_2=\dS$.
Otherwise, we have $j\geq 2$. Let $x$ be the maximum $\ls(g_{j-1},f)$ in $\<$ over all $f\in \dS$,
and let $\dS_1$ be the set of all $f\in \dS$ containing the subpath $u_0Lx$.
Next, let $y$ be the minimum $\rs(g_{j-1},f')$ in $\<$ over all $f'\in \dS_1$,
and then let $\dS_2$ be the set of all $f'\in \dS_1$ containing the subpath $yRv_0$.

\item[(b).] If there exists some paths in $\dS_2$ containing some transforming edge $e$ in $L_\dP$ with $x\pq V(e)$,
then let $\dS_3$ be the set consisting of all such paths in $\dS_2$; otherwise, let $\dS_3=\dS_2$.
Next, if there exists some paths in $\dS_3$ containing some transforming edge $e'$ in $R_\dP$ with $V(e')\pq y$,
then let $\dS_4$ be the set consisting of all these paths in $\dS_3$; otherwise, let $\dS_4=\dS_3$.

\item[(c).] Pick any path in $S_4$ and denote it by $g_j$. 
Update $j \leftarrow j+1 $ and $\dS \leftarrow \dS\bh \{g_j\}$.
\end{itemize}

We also need some properties on $g_j$'s, which can be collected directly from the above algorithm.

\begin{prop}\label{prop:from-alg}
The following hold for any $j<k<\ell$:
\begin{itemize}
\item [(i).] If both $g_j$ and $g_\ell$ contain some subpath $u_0Lw$, then $g_k$ also contains $u_0Lw$. 
\item [(ii).] $\ls(g_{j-1},g_j)\sq \ls(g_{j-1},g_k)$, and if $\ls(g_{j-1},g_j)=\ls(g_{j-1},g_k)$ then $\rs(g_{j-1},g_j)\pq \rs(g_{j-1},g_k)$.
\item [(iii).] Suppose that $\ls(g_{j-1},g_j)=\ls(g_{j-1},g_k)$ and $\rs(g_{j-1},g_j)=\rs(g_{j-1},g_k)$.
If $g_j$ contains no transforming edge in $L_\dP$, then $g_k$ also contains no transforming edge in $L_\dP$.
\item [(iv).] Under the same conditions of (iii),
if $g_k$ contains a transforming edge in $R_\dP$ and $g_j$ does not, then $g_j$ contains a transforming edge in $L_\dP$ but $g_k$ does not.
\end{itemize}
\end{prop}

We then partition $\dP=\{g_j\}_{j\geq 1}$ into a bounded number of subsets in the next definition.

\begin{dfn}
For any transforming edge $e$ of $\dP$, the path $g_j\in \dP$ containing $e$ with the minimum $j$ is called a {\it fence} of $\dP$.
Let $\{g_{j_k}\}_{1\leq k< k_\dP}$ be the set of all fences of $\dP$, where the sequence $\{j_k\}$ is increasing with $k$.
Then the set $\dI_k=\{g_j: j_k\leq j< j_{k+1}\}$ for each $0\leq k< k_\dP$ is called an {\it interval} of $\dP$,
where we define $j_0=1$ and $j_{k_\dP}=|\dP|+1$.
\end{dfn}

So $\dP$ is partitioned into at most $k_\dP$ intervals.
By Proposition \ref{prop:d(e)>gamma}, any path in $\dP$ has at most one leaf-edge in $L_\dP$ and at most one leaf-edge in $R_\dP$.
This shows that $k_\dP\leq 2|\dP|/\gamma\leq 3\sqrt{n}/\gamma$, as desired.


Towards Lemma \ref{lem: consecutive}, we first prove the following weaker version.
Let us recall the definitions of feasible triples and quadruples, stated right before Subsection~\ref{subsec:feasible-cycles}.

\begin{prop}\label{prop: 101}
For any edge $e$ in $G(\dP)$, let $g_i$ be the path in $\dP$ containing $e$ with minimum $i$.
Then one can delete at most $2\beta$ paths from $\dP$ such that for any $k>j>i$,
if some remaining path $g_j$ does not contain $e$, then every remaining path $g_k$ does not contain $e$.
\end{prop}
\begin{proof}
Suppose this fails for some edge $xy$ in $G(\dP)$ with $x\<y$.
Let $g_{i}$ be the path in $\dP$ containing $xy$ with minimum $i$.
We may assume that $d(xy)>2\beta$ and there are at least $\beta$ paths in $\dP$ behind $g_{i}$ which does not contain $xy$.
Let $\dB$ be the set of the first $\beta$ such paths.
Note that paths in $\dB$ may not be consecutive in $\dP$.
Let $g_{j_0}$ be the path in $\dB$ with maximum $j_0$.
Let $\dC$ be the set consisting of all paths in $\dP$ behind $g_{j_0}$ and containing $xy$,
and we may also assume that $|\dC|\geq \beta$.

Let $g_j$ be the path in $\dB$ with minimum $j$.
Then $g_j$ does not contain $xy$, while $g_{j-1}$ does. 
As $\beta\geq \alpha_0\geq n^{1/4}\log n$,
we see that there are at least $\beta-2n^{1/4}\geq \beta/2$ paths $g_k$ in $\dC$
such that $\{g_{j-1}, g_k\}$ and $\{g_j, g_k\}$ are normal pairs. 
From now on, by $g_k$ we mean any one of such paths in $\dC$.

Let $a\<b$ be the two splitting vertices in $g_{j-1}\cup g_k$ (see Figure 4-a).
Since $xy\in E(g_{j-1})\cap E(g_k)$, we have either $y\pq a$ or $b\pq x$.
If $y\pq a$, then $xy$ is in the tree $L$ and thus both $g_{j-1}$ and $g_k$ contain the subpath $u_0Ly$, while $g_j$ does not.
This contradicts Proposition \ref{prop:from-alg} (i).
Hence, $b\pq x$.
Since $xy\notin E(g_j)$, we see that $\rs(g_{j-1},g_j)$ and $\rs(g_k,g_j)$ are the same vertex, say $z$, with $b\pq x\<y\pq z$.

\begin{figure}[ht!]\label{fig:A}
\begin{minipage}{0.5 \linewidth}
        \centering
\begin{tikzpicture}[scale=0.8,line cap=round,line join=round,>=triangle 45,x=1cm,y=1cm]
\clip(-0.75,-1) rectangle (8,2);
\draw [line width=1pt] (0,0)-- (1,0);
\draw [line width=1pt] (4.5,0)--(7.5,0);
 \draw [line width=1pt] (1,0 ) to[out=90,in=90]  (4.5,0);

\draw [line width=1pt] (1,0 ) to[out=-50,in=-140]  (4.5,0);
\draw [line width=0.5pt] (3.5,0.9 ) to[out=20,in=110]  (6,0);

 \draw [line width=0.5pt] (1.5,0.74)  .. controls (1.75,2) and  ( 2.5,2) ..  (3,1.03);

\begin{scriptsize}
\draw [fill=black] (0,0) circle (1.5pt);
\draw[color=black] (-0.25 ,0) node {$u$};
\draw [fill=black] (7.5,0) circle (1.5 pt);
\draw[color=black] (7.5,-0.25) node {$v$};

\draw [fill=black] (1,0) circle (1.5pt);
\draw[color=black] (1.3 ,0) node {$a$};
\draw [fill=black] (4.5,0) circle (1.5pt);
\draw[color=black] (4.5,-0.25) node {$b$};

\draw [fill=black] (1.5,0.74) circle (1.5pt);
\draw[color=black] (1.7,0.6) node {$w'$};
\draw [fill=black] (3,1.03) circle (1.5pt);
\draw[color=black] (3,0.8) node {$x'$};
\draw [fill=black] (3.5,0.9 ) circle (1.5pt);
\draw[color=black] (3.5,0.7) node {$y'$};
\draw [fill=black] (6,0) circle (1.5pt);
\draw[color=black] (6 ,-0.25) node {$z$};

\draw [fill=black] (5,0) circle (1.5pt);
\draw[color=black] (5 ,-0.25) node {$x$};
\draw [fill=black] (5.4,0) circle (1.5pt);
\draw[color=black] (5.4,-0.25) node {$y$};

\draw[color=black] (0.6,0.3) node {$g_{j-1}$};
\draw[color=black] (0.4, -0.3) node {$g_k$};

\draw[color=black] (3 ,1.6) node {$g_j$};
\draw[color=black] (4.2 ,1.3 ) node {$g_j$};
\end{scriptsize}
\end{tikzpicture}
\caption*{(a)}
\end{minipage}
\begin{minipage}{0.5 \linewidth}
        \centering
\begin{tikzpicture}[scale=0.8,line cap=round,line join=round,>=triangle 45,x=1cm,y=1cm]
\clip(-0.75,-1) rectangle (8,2);
\draw [line width=1pt] (0,0)-- (1,0);
\draw [line width=1pt] (4.5,0)--(7.5,0);
 \draw [line width=1pt] (1,0 ) to[out=90,in=90]  (4.5,0);

\draw [line width=1pt] (1,0 ) to[out=-50,in=-140]  (4.5,0);
\draw [line width=0.5pt] (3.5,0.9 ) to[out=20,in=110]  (6,0);

\draw [line width=0.5pt] (1.5,0.74) to[out=90,in=120]  (3,1.03);
\draw  [line width=0.8pt,dash pattern=on 2pt off 2pt] (1,0)  .. controls (1.1,3) and  ( 6,2) ..  (6.8,0);
\begin{scriptsize}
\draw [fill=black] (0,0) circle (1.5pt);
\draw[color=black] (-0.25 ,0) node {$u$};
\draw [fill=black] (7.5,0) circle (1.5 pt);
\draw[color=black] (7.5,-0.25) node {$v$};

\draw [fill=black] (1,0) circle (1.5pt);
\draw[color=black] (1.6 ,0) node {$a(a')$};
\draw [fill=black] (4.5,0) circle (1.5pt);
\draw[color=black] (4.5,-0.25) node {$b$};

\draw [fill=black] (1.5,0.74) circle (1.5pt);
\draw[color=black] (1.7,0.6) node {$w$};
\draw [fill=black] (3,1.03) circle (1.5pt);
\draw[color=black] (3,0.8) node {$x'$};
\draw [fill=black] (3.5,0.9 ) circle (1.5pt);
\draw[color=black] (3.5,0.7) node {$y'$};
\draw [fill=black] (6,0) circle (1.5pt);
\draw[color=black] (6 ,-0.25) node {$z$};

\draw [fill=black] (6.8,0) circle (1.5pt);
\draw[color=black] (6.8,-0.25)node {$z'$};

\draw [fill=black] (5,0) circle (1.5pt);
\draw[color=black] (5 ,-0.25) node {$x$};
\draw [fill=black] (5.4,0) circle (1.5pt);
\draw[color=black] (5.4,-0.25) node {$y$};

\draw[color=black] (0.6,0.3) node {$g_{j-1}$};
\draw[color=black] (0.4, -0.3) node {$g_k$};
\draw[color=black] (1.2 ,1.3) node {$g_{j'}$};
\draw[color=black] (3 ,1.4) node {$g_j$};
\draw[color=black] (4.2 ,1.3) node {$g_j$};
\end{scriptsize}
\end{tikzpicture}
\caption*{(b)}
\end{minipage}
\caption*{Figure 4}
\end{figure}

Let $w=\ls(g_{j-1},g_j)$. We claim that $a\< w\< b$.
Note that $\{g_j,g_k\}$ is a normal pair and by Proposition \ref{prop:from-alg} (ii), $w\sq ~\ls(g_{j-1},g_k)=a$.
If $w\sq b$, then we have $g_j=ug_{j-1}w\cup wg_jz\cup zg_{j-1}v$,
which implies that $\{g_j,g_k\}$ is type-\2, a contradiction.
So $a\pq w\<b$. Now assume that $w=a$.
Let $w'=\ls(g_{j}, g_k)$. So $w'\in V(ag_kb)\bh \{b\}$ and $g_j=ug_kw'\cup w'g_jz\cup zg_kv$.
We see $a=\ls(g_{j-1},g_j)=\ls(g_{j-1}, g_k)$, but $\rs(g_{j-1},g_j)=z\s b=\rs(g_{j-1},g_k)$, a contradiction to Proposition \ref{prop:from-alg} (ii).
This proves the claim that $a\< w\< b$.
Now $a,z$ are the only two splitting vertices in $g_j\cup g_k$.

If $\{g_{j-1}, g_j\}$ is not type-\1, then $g_j=ug_{j-1}w\cup wg_jz\cup zg_{j-1}v$ and further,
we see that $g_{j-1}$ is a crossing path of $\{g_j,g_k\}$, contradicting that $\{g_j,g_k\}$ is a normal pair.
Thus, $\{g_{j-1}, g_j\}$ is type-\1.

Since $a,z$ are the only two splitting vertices in $g_j\cup g_k$,
we may assume that the splitting vertices in $g_{j-1}\cup g_j$ are $w,x',y',z$ such that $a\prec w\prec x'\preceq  y'\prec b\prec z$ lie in $g_{j-1}$.
We write $f_{\pi(t)}=g_t$ for each $t\in \{j-1, j, k\}$.
If $P_{\pi(j)}\subseteq yg_jz$, then $\{g_j,g_k\}$ is type-\2 with a crossing path $g_{j-1}$, a contradiction.
So by Proposition~\ref{prop:2cycleI}, we can deduce that $P_{\pi(j)}\subseteq wg_jx'$ and $P_{\pi(j-1)}\subseteq y'g_{j-1}z$.
In particular, $y'g_{j-1}b$ contains some edge in $P_{\pi(j-1)}$.
Let the base of $\{g_{j-1}, g_j\}$ be $f_\ell$. 
Then Proposition~\ref{prop:2cycleI} also shows that $y$ is an inner vertex of $P_\ell$ and thus $y'f_\ell z=y'g_jz$ contains some edge in $P_\ell$.

Recall the set $\dB$. For any chosen $g_k\in \dC$ from above,
we can find at least $\beta-3n^{1/4}\geq \beta/2$ paths $g_{j'}\in \dB$ such that $g_{j'}$ forms a normal pair with any path in $\{g_{j-1},g_j,g_k\}$.
Notice that we do not require $\{g_{j'},f_\ell\}$ to be a normal pair as $f_\ell$ may not be in $\dP$.
Let $z'=\rs(g_{j'},g_{j-1})$.
As $g_{j'}$ does not contain $xy$, we see $z'\succeq y$ and thus $z'$ is also the vertex $\rs(g_{j'},g_k)$.
Since $\{g_{j'},g_{j-1}\}$ and $\{g_{j'},g_k\}$ both are normal,
$g_{j'}\bh (g_{j-1}\cup g_k)$ forms a path, say $a'g_{j'}z'$, where $a'\in V(u_0g_{j-1}x\cup ag_kb)$. See Figure 4-b for an illustration.

We claim that $a'=a$.
Since $j<j'<k$ and both $g_{j-1}$ and $g_k$ contain $ug_{j-1}a=uLa$,
by Proposition \ref{prop:from-alg}, $g_{j'}$ also contains $ug_{j-1}a$, which implies that $a'\sq a$.
If $a'\sq b$, then $g_{j'}=ug_{j-1}a'\cup a'g_{j'}z'\cup z'g_{j-1}v$ and $\{g_{j'},g_k\}$ forms a type-\1 pair, a contradiction.
If $a'\in V(ag_{j-1}b)\bh \{a,b\}$, then $\{g_{j'},g_k\}$ is type-\2 with a crossing path $g_{j-1}$, a contradiction.
If $a'\in V(ag_{k}b)\bh \{a,b\}$, then $\{g_{j'},g_{j-1}\}$ is type-\2 with a crossing path $g_{k}$, again a contradiction.
Therefore we have $a'=a$.

Next we claim that $ag_{j'}z'$ is internally disjoint with $y'g_jz$.
Suppose on the contrary that there exists a splitting vertex $w'\in V(ag_{j'}z')\cap V(y'g_jz)$ in $g_{j'}\cup g_j$.
Since $\{g_{j'},g_j\}$ is normal, we have $z'=z$ and $w'g_{j'}z=w'g_jz$.
However, as $a\<y'\<w'\<z$ lie in $g_j$, we see that $g_j$ is a crossing path of $\{g_{j'},g_{j-1}\}$, a contradiction.
This proves the claim.

Recall that $y'g_{j-1}b$ contains some edge in $P_{\pi(j-1)}$ and $y'f_\ell z=yg_jz$ contains some edge in $P_\ell$.
Thus, $ag_kb\cup bg_{j-1}y' \cup y'f_\ell z \cup zg_kz' \cup z'g_{j'}a$ is a 4-feasible cycle for $\{g_{j-1}, g_{j'}, g_k, f_\ell\}$.

Putting everything together, there are at least $\beta/2$ choices of $g_k\in \dC$ and
subject to a fixed $g_k$, there are at least $\beta/2$ choices of $g_{j'}\in \dB$
such that $\{g_{j-1}, g_k, g_{j'}\}$ is contained in a feasible quadruple.
That is, $g_{j-1}\in \dP$ is contained in at least $\beta^2/4\geq \sqrt{n}\log^2n/4>26\sqrt{n}\log n$ distinct $\dW_{pq}$'s.
This final contradiction (to the definition of $\dP$) completes the proof of Proposition~\ref{prop: 101}.
\end{proof}

Finally, we are ready to complete the proof of Lemma \ref{lem: consecutive}.

\begin{proof}[\bf Proof of Lemma \ref{lem: consecutive}.]
Putting everything above together, it suffices for us to prove the following statement.
For any edge $e$ in $G(\dP)$ with $d(e)\geq \gamma$,
one can delete at most $3\beta$ paths in $\dP$ such that
there is at most one interval of $\dP$ which can contain some remaining paths $g_j, g_k$ with $e\in E(g_j)$ and $e\notin E(g_k)$ and moreover, all such paths $g_j, g_k$ satisfy $j<k$.

Suppose this fails for some edge $e$ with $d(e)\geq \gamma$.
Let $g_i$ be the path in $\dP$ containing $e$ with minimum $i$.
By Proposition~\ref{prop: 101}, one can delete at most $2\beta$ paths from $\dP$ such that for any $k>j>i$,
if some remaining path $g_j$ does not contain $e$, then every remaining path $g_k$ does not contain $e$.
If $g_i$ is a fence, then only the last interval of $\dP$, which has some remaining path containing $e$, can contain some remaining path $g_\ell$ with $e\notin E(g_\ell)$,
and the conclusion holds.
Therefore, we may assume that $g_i$ is not a fence.
It will be enough for us to show that by deleting extra $\beta$ paths,
every remaining path in the interval containing $g_i$ contains the edge $e$.

Let $\dA$ be the set of all paths before $g_i$ in this interval.
We may assume that $|\dA|\geq \beta$.
By Proposition \ref{prop:d(e)>gamma}, $e$ is in either $L_\dP$ or $R_\dP$. We consider two cases (see Figure 5 for an illustration).

\medskip

\noindent {\bf Case A. $e$ is in $L_\dP$.}

\medskip

We obverse that $g_i$ contains no transforming edge which lies below $e$ in $L_\dP$
(as otherwise, $g_i$ would be the first path in $\dP$ containing this transforming edge and thus become a fence).
There exists at least one transforming edge $e'$ in $L_\dP$ such that $e$ lies in the subpath of $L_\dP$ between $u_0$ and $e'$.
Let $\dB$ be the set of all paths in $\dP$ containing $e'$.
Then $|\dB|\geq \gamma$ and all paths in $\dB$ also contain $e$.
There are at least $\gamma-2n^{1/4}\geq \beta$ paths $g_k\in \dB$ such that $\{g_k,g_{i-1}\}$ and $\{g_k,g_{i}\}$ are normal.
From now on, we fix such a path $g_k$.
Note that $g_{i-1}$ does not contain $e$, while $g_i$ and $g_k$ contain $e$.

Let $a, b$ be the two splitting vertices in $g_i\cup g_k$.
Clearly $V(e)\pq a\<b$.
Also we see that $\ls(g_{i-1},g_i)$ and $\ls(g_{i-1},g_k)$ are the same vertex, say $w$, such that $w\pq V(e)$.
We make the following claim.

\medskip

{Claim A.} Let $z=\rs(g_{i-1},g_i)$.
Then $z$ is an inner vertex in $ag_ib$.

\begin{proof}[Proof of Claim A]
Let $z'=\rs(g_{i-1},g_k)$. If $z'\<b$, then clearly $z'g_{i-1}v=z'g_kv$, implying that $z'\<b=z$, a contradiction to Proposition \ref{prop:from-alg} (ii).
So $z'\succeq b$.
If $z\succeq b$, then we see $z=z'$.
In this case, $g_k$ contains some transforming edge (that is $e'$) in $L_\dP$, while $g_i$ does not, a contradiction to Proposition \ref{prop:from-alg} (iii).
If $z\pq a$, then $wg_{i-1}v=wg_{i-1}z\cup zg_iv$ and thus $\{g_{i-1}, g_k\}$ must be type-\1, a contradiction to the choice of $g_k$.
Hence, $z$ is an inner vertex in $af_ib$ (and also $\rs(g_{i-1},g_k)=b$).
\end{proof}

If $\{g_{i-1},g_i\}$ is not type-\1, then $g_{i-1}\bh (g_i\cup g_k)$ is a subpath $wg_{i-1}z$,
which shows that $\{g_{i-1},g_k\}$ is type-\2 with a crossing path $g_i$, a contradiction.

Therefore, $\{g_{i-1},g_i\}$ is type-\1 with base $f_\ell$ and splitting vertices $w\<x\pq y\<z$.
Since $w,b$ are the only two splitting vertices in $g_{i-1}\cup g_k$,
we see that $w\< a\< x\pq y\< z\< b$ lie in $g_i$.
Let $f_{\pi(t)}=g_t$ for every $t\in \{i-1, i, k\}$.
If $P_{\pi(i-1)}\subseteq wg_{i-1}x$, then again $\{g_{i-1},g_k\}$ is type-\2 with a crossing path $g_i$.
So by Proposition \ref{prop:2cycleI}, $P_{\pi(i)}\subseteq wg_ix$, $P_{\pi(i-1)}\subseteq yg_{i-1}z$,
and $x, y$ are inner vertices of $P_\ell$.
This shows that
\begin{align}\label{equ:Lem4.2}
\mbox{$ag_ix$ contains some edge in $P_{\pi(i)}$ and $xg_{i-1}w=xf_\ell w$ contains some edge in $P_\ell$.}
\end{align}
There exist at least $\beta-3n^{1/4}-1\geq \beta/2$ paths $g_j$ in $\dA\bh\{f_\ell\}$
such that $g_j$ forms a normal pair with any path in $\{g_{i-1},g_i,g_k\}$.
Any such path $g_j$ does not contain the edge $e$.
So we can set $w'=\ls(g_{j},g_i)=\ls(g_{j},g_k)$ such that $w'\pq V(e)$.
Then as $g_j$ forms a normal pair with each of $g_i$ and $g_k$,
we can infer that $g_{j}\bh (g_i\cup g_k)$ is a path,
say $w'g_{j}z'$ with $z'\in V(y'g_iv\cup ag_kb)$, where we write $e=x'y'$ with $x'\<y'$.
Note that $z'\in V(y'g_ia)$ is impossible, as otherwise we can deduce that $ag_jv=ag_iv=ag_kv$, a contradiction.
If $z'\in V(ag_kb)\bh \{a,b\}$, then $\{g_{j}, g_i\}$ is type-\2 with a crossing path $g_k$, a contradiction.
If $z'\in V(ag_ib)\bh \{a,b\}$, then $\{g_{j}, g_k\}$ is type-\2 with a crossing path $g_i$, a contradiction.
So we have $z'\in V(bg_iv)$ and $g_j=ug_iw'\cup w'g_jz'\cup z'g_iv$.
Let $f_{\pi(j)}=g_j$.
Whenever the two paths $w'g_jz'$ and $wg_{i-1}x=wf_\ell x$ intersect or not (if they do then $w=w'$),
$xf_\ell w\cup wg_iw' \cup w'g_jz'$ contains an $(x,z')$-path $Q$ which contains some edge in $P_\ell$ and some edge in $P_{\pi(j)}$.
By \eqref{equ:Lem4.2},
we see that $Q\cup z'g_ka\cup ag_ix$ contains a 4-feasible cycle for $\{g_i,g_{j},g_k,f_\ell\}$,

Putting all together, there are at least $\beta$ choices of $g_k\in \dB$ and
subject to a fixed $g_k$, there are at least $\beta/2$ choices of $g_j\in \dB$
such that $\{g_{i}, g_j, g_{k}\}$ is contained in a feasible quadruple.
So $g_i$ in $\dP$ is contained in at least $\beta^2/2>26\sqrt{n}\log n$ distinct $\dW_{jk}$'s.
This contradicts the definition of $\dP$ and completes the proof of Case A.
\begin{figure}[ht!]\label{fig:A}
\begin{minipage}{0.33 \linewidth}
        \centering
\begin{tikzpicture}[scale=0.7,line cap=round,line join=round,>=triangle 45,x=1cm,y=1cm]
\clip(-0.75,-1.5) rectangle (7.75,2 );
\draw [line width=1pt] (0,0)-- (3,0);
\draw [line width=1pt] (5,0)-- (7,0);

\draw [line width=1pt] (3,0 ) .. controls (3.4,1 ) and  ( 4.6,1 ) ..  ( 5,0);

\draw [line width=1pt] (3,0 ) to[out=-70,in=-110]  (5,0);
\draw [line width=0.8pt,dash pattern=on 2pt off 2pt] (1.5,0 ) to[out=-70,in=-110]  (6,0);

\draw [line width=0.5pt] (0.8 ,0 )  to[out=50,in=150]    (3.5,0.62);

\draw [line width=0.5pt] (4,0.75)   .. controls (4.2,0.9 ) and (4.68,1.2 )   ..   (4.7,0.46);
\begin{scriptsize}
\draw [fill=black] (0,0) circle (1.5pt);
\draw[color=black] (0 ,-0.25) node {$u$};
\draw [fill=black] (7,0) circle (1.5 pt);
\draw[color=black] (7.1 ,-0.25) node {$v$};

\draw [line width=2pt] (2.0,0)-- (2.4,0);
\draw [fill=black] (2.0,0) circle (1.5 pt);
\draw [fill=black] (2.4,0) circle (1.5 pt);
\draw[color=black] (2.2 ,-0.25) node {$e$};

\draw [line width=2pt] (3.8,-0.535)-- (4.2,-0.535);
\draw [fill=black] (3.8,-0.535) circle ( 1.5 pt);
\draw [fill=black] (4.2,-0.535) circle ( 1.5 pt);
\draw[color=black] (4 ,-0.77) node {$e'$};

\draw [fill=black] (3,0) circle (1.5pt);
\draw[color=black] (2.9 ,-0.25) node {$a$};
\draw [fill=black] (5,0) circle (1.5pt);
\draw[color=black] (5.1 ,-0.25) node {$b$};

\draw [fill=black] (0.8,0) circle (1.5pt);
\draw[color=black] (0.8 ,-0.25) node {$w$};
\draw [fill=black] (1.5,0) circle (1.5pt);
\draw[color=black] (1.5, 0.3) node {$w'$};

\draw [fill=black] (3.5,0.62) circle (1.5pt);
\draw[color=black] (3.5 ,0.4) node {$x$};
\draw [fill=black] (4,0.75) circle ( 1.5pt);
\draw[color=black] (4 ,0.5) node {$y$};
\draw [fill=black] (4.7,0.46) circle ( 1.5pt);
\draw[color=black] (4.6 , 0.25) node {$z$};

\draw [fill=black] (6,0) circle (1.5pt);
\draw[color=black] (5.8,  0.3) node {$z'$};

\draw[color=black] (2,1.1) node {$g_{i-1}$};
\draw[color=black] (4.4 ,1.2) node {$g_{i-1}$};
\draw[color=black] (6.5, 0.3) node {$g_i$};
\draw[color=black] (6.5, -0.3) node {$g_k$};
\draw[color=black] (2,-1.3) node {$g_j$};
\end{scriptsize}
\end{tikzpicture}
\caption*{\footnotesize Case A: $z'\in V(bg_iv)$}
\end{minipage}
\begin{minipage}{0.33\linewidth}
        \centering
\begin{tikzpicture}[scale=0.7,line cap=round,line join=round,>=triangle 45,x=1cm,y=1cm]
\clip(-1.75,-1.5) rectangle (6.75,2 );
\draw [line width=1pt] (-1,0)-- (3,0);
\draw [line width=1pt] (5,0)-- (6,0);

\draw [line width=1pt] (3,0 ) .. controls (3.4,1 ) and  ( 4.6,1 ) ..  ( 5,0);
\draw [line width=1pt] (3,0 ) to[out=-70,in=-110]  (5,0);
\draw [line width=0.8pt,dash pattern=on 2pt off 2pt] (0.5,0 )  .. controls (1,-1  ) and (2,-1  )   ..  (2.5,0);  
\draw [line width=0.5pt] (-0.2 ,0 )  to[out=50,in=150]    (3.5,0.62);
\draw [line width=0.5pt] (4,0.75)   .. controls (4.2,0.9 ) and (4.68,1.2 )   ..   (4.7,0.46);
\begin{scriptsize}
\draw [fill=black] (-1,0) circle (1.5pt);
\draw[color=black] (-1 ,-0.25) node {$u$};
\draw [fill=black] (6,0) circle (1.5 pt);
\draw[color=black] (6.1 ,-0.25) node {$v$};

\draw [line width=2pt] (1.3,0)-- (1.7,0);
\draw [fill=black] (1.3,0) circle (1.5 pt);
\draw [fill=black] (1.7,0) circle (1.5 pt);
\draw[color=black] (1.5 ,-0.25) node {$e$};

\draw [line width=2pt] (3.8,-0.535)-- (4.2,-0.535);
\draw [fill=black] (3.8,-0.535) circle ( 1.5 pt);
\draw [fill=black] (4.2,-0.535) circle ( 1.5 pt);
\draw[color=black] (4 ,-0.77) node {$e'$};

\draw [fill=black] (3,0) circle (1.5pt);
\draw[color=black] (2.9 ,-0.25) node {$a$};
\draw [fill=black] (5,0) circle (1.5pt);
\draw[color=black] (5 ,-0.25) node {$b$};

\draw [fill=black] (-0.2,0) circle (1.5pt);
\draw[color=black] (-0.2 ,-0.25) node {$w$};
\draw [fill=black] (0.5,0) circle (1.5pt);
\draw[color=black] (0.5, 0.3) node {$w'$};
\draw [fill=black] (2.5,0) circle (1.5pt);
\draw[color=black] (2.5, 0.3) node {$z'$};

\draw [fill=black] (3.5,0.62) circle (1.5pt);
\draw[color=black] (3.5 ,0.4) node {$x$};
\draw [fill=black] (4,0.75) circle ( 1.5pt);
\draw[color=black] (4 ,0.5) node {$y$};
\draw [fill=black] (4.7,0.46) circle ( 1.5pt);
\draw[color=black] (4.6 , 0.25) node {$z$};

\draw[color=black] (2,1.3) node {$g_{i-1}$};
\draw[color=black] (4.4 ,1.3) node {$g_{i-1}$};
\draw[color=black] (5.5, 0.3) node {$g_i$};
\draw[color=black] (5.5, -0.3) node {$g_k$};
\draw[color=black] (1.5,-1 ) node {$g_j$};
\end{scriptsize}
\end{tikzpicture}
\caption*{\footnotesize Case A: $z'\in V(y'g_ia)$}
\end{minipage}
\begin{minipage}{0.33 \linewidth}
        \centering
\begin{tikzpicture}[scale=0.7,line cap=round,line join=round,>=triangle 45,x=1cm,y=1cm]
\clip(-0.75,-1.5) rectangle (7.75,2 );
\draw [line width=1pt] (0,0)-- (2,0);
\draw [line width=1pt] (4,0)-- (7,0);

\draw [line width=1pt] (2,0 ) .. controls (2.4,1 ) and  ( 3.6,1 ) ..  ( 4,0);

\draw [line width=1pt] (2,0 ) to[out=-70,in=-110]  (4,0);

\draw [line width=0.5pt] (1,0 )  to[out=50,in=150]    (2.3,0.46);

 \draw [line width=0.5pt] (3.7,0.46)   .. controls (4,1.1 ) and (5.8,1.2 )   ..   (6,0);
\begin{scriptsize}
\draw [fill=black] (0,0) circle (1.5pt);
\draw[color=black] (0 ,-0.25) node {$u$};
\draw [fill=black] (7,0) circle (1.5 pt);
\draw[color=black] (7.1 ,-0.25) node {$v$};

\draw [fill=black] (2,0) circle (1.5pt);
\draw[color=black] (1.9 ,-0.25) node {$a$};
\draw [fill=black] (4 ,0) circle (1.5pt);
\draw[color=black] (4.1 ,-0.25) node {$b$};

\draw [line width=2pt] (4.8,0)-- (5.2,0);
\draw [fill=black] (4.8,0) circle (1.5pt);
\draw [fill=black] (5.2,0) circle (1.5pt);
\draw[color=black] (5 ,-0.25) node {$e$};

\draw [line width=2pt] (2.8,0.75)-- (3.2,0.75);
\draw[color=black] (3 , 0. 5) node {$e_1$};
\draw [fill=black] (2.8,0.75) circle ( 1.5pt);
\draw [fill=black] (3.2,0.75) circle ( 1.5pt);

\draw [fill=black] (1,0) circle (1.5pt);
\draw[color=black] (1 ,-0.25) node {$w$};

\draw [fill=black] (2.3,0.46) circle (1.5pt);
\draw [fill=black] (3.7,0.46) circle ( 1.5pt);

\draw [fill=black] (6,0 ) circle ( 1.5pt);
\draw[color=black] (6 , -0.25) node {$z$};

\draw[color=black] (1.5,0.8) node {$g_{i-1}$};
\draw[color=black] (5 ,1.2) node {$g_{i-1}$};
\draw[color=black] (6.5, 0.3) node {$g_i$};
\draw[color=black] (6.5, -0.3) node {$g_k$};
\end{scriptsize}
\end{tikzpicture}
\caption*{\footnotesize Case B: A key step for Claim B}
\end{minipage}
\caption*{Figure 5}
\end{figure}
\medskip

\noindent {\bf Case B. $e$ is in $R_\dP$.}

\medskip
The proof of this case is similar to that of Case A.
We also see that $g_i$ contains no transforming edge which lies in $R_\dP$ below $e$ (as otherwise $g_i$ is a fence).
So there exists a transforming edge $e_0$ in $R_\dP$ such that $e$ lies in the path of $R_\dP$ between $v_0$ and $e_0$.
Let $\dC$ be the set of all paths in $\dP$ containing $e_0$.
Then $|\dC|\geq \gamma$ and all paths in $\dC$ also contain $e$.
There are at least $\gamma-2n^{1/4}\geq \beta$ paths $g_k\in \dC$ such that $\{g_k,g_{i-1}\}$ and $\{g_k,g_i\}$ are normal.
Let $a,b$ be the only two splitting vertices in $g_k\cup g_i$.
Note that $g_i$ and $g_k$ contain $e$, while $g_{i-1}$ does not.
So $\rs(g_{i-1},g_i)$ and $\rs(g_{i-1},g_k)$ are the same vertex, say $z$, such that $a\<b\pq V(e)\pq z$.

We need the following claim, which plays the parallel role as Claim A in the previous case.

\medskip

{Claim B.} Let $w=\ls(g_{i-1},g_i)$.
Then $w$ is an inner vertex in $ag_ib$.

\begin{proof}[Proof of Claim B]
First suppose that $b\pq w$. Since $ug_{i-1}w=ug_iw$ contains $ag_ib$,
we see that $\{g_{i-1},g_k\}$ must be a type-\1 pair, a contradiction to the choice of $g_k$.

To prove Claim B, it suffices to consider when $w\pq a$ (see Figure 5).
Let $w'=\ls(g_{i-1},g_k)$.
If $w'\s a$, then clearly $ug_{i-1}w'=ug_kw'$, implying that $w=a\< w'$, a contradiction to Proposition \ref{prop:from-alg} (ii).
So $w'\pq a$.
In this case, we see $w=w'\in V(ug_ia)$, that is, $g_i$ and $g_k$ share the same first and last splitting vertices with $g_{i-1}$.
Since $g_k$ contains a transforming edge in $R_\dP$ while $g_i$ does not,
by Proposition \ref{prop:from-alg} (iv),
$g_i$ contains a transforming edge (say $e_1$) in $L_\dP$ and $g_k$ does not.
Clearly, such $e_1$ is in $ag_ib$.
If $g_{i-1}$ does not contain $e_1$,
then $g_i$ is the first path in $\dP$ containing $e_1$ and thus $g_i$ is a fence, a contradiction.
Hence, $g_{i-1}$ also contains $e_1\in E(L_\dP)$.
By Proposition \ref{prop:d(e)>gamma}, $g_{i-1}$ and $g_i$ contains all edges in the subpath of $L$ from $u$ to $e_1$.
This indicates that $\ls(g_{i-1},g_i)=w\s a$, a contradiction.
Therefore, $w$ must be an inner vertex in $ag_ib$.
\end{proof}

The remaining proof is analogous to Case A.
In fact, once we are equipped with Claims A and B, these two cases are identical if one revises the linear ordering $\<$.
For simplicity, we omit the detailed verification of the remaining proof of Case B here.
We finish the proof of Lemma \ref{lem: consecutive}.
\end{proof}


\section{Proof of the main result}
Now we are ready to prove Theorem \ref{thm: main}, by using a counting strategy motivated by \cite{Lin}.
Let $n$ be sufficiently large and $G$ be an $n$-vertex 2-connected graph with $n+s$ edges which does not contain two cycles of the same length.
Suppose for a contradiction that $s\geq (1+o(1))\sqrt{n}$.

Let $\dP_1, \dP_2, \dP_3, \dP_4$ be subsets of $\dF$ and $\beta, \gamma$ be constants from Lemma \ref{lem: consecutive}.
So $\beta\geq n^{1/4}\log n$, $\gamma=\beta \log n\leq \sqrt{n}$, $s':=\sum_{i\in [4]} |\dP_i|\geq (1-o(1))s$, and each $\dP_i$ has an arrangement $\{g_j\}_{j\geq 1}$.

Let $\Phi$ be the set of all normal pairs $\{g_j,g_k\}$ such that $g_j,g_k$ are from the same interval of $\dP_i$ for some $i\in [4]$ and $\beta\leq |j-k|\leq \sqrt{\beta\gamma}$.
For $e\in E(G)$ and $\{g_j,g_k\}\in \Phi$, let $\delta(e,g_j,g_k)$ be an index function such that $\delta=1$ if exactly one of the paths $g_j,g_k$ contains $e$ and $\delta=0$ otherwise.

In the coming proof, we are estimating the summation $\Sigma$ of $\delta(e,g_j,g_k)$ over all $e\in E(G)$ and $\{g_j,g_k\}\in \Phi$.
First, let us bound the size of $\Phi$.
The total number of pairs $\{g_j,g_k\}$ such that $g_j, g_k$ lie in some $\dP_i$ and $\beta\leq |j-k|\leq \sqrt{\beta\gamma}$
is at least $\sum^{\sqrt{\beta\gamma}}_{r=\beta}(s'-4r)\geq (1-o(1))s\sqrt{\beta\gamma}=(1-o(1))s\beta\sqrt{\log n},$
where the inequality holds because $s'=\sum_{i\in [4]} |\dP_i|\geq (1-o(1))s$ and every $r\leq \sqrt{\beta\gamma}=\gamma/\sqrt{\log n}\leq \sqrt{n/\log n}=o(s)$.
By Lemma \ref{lem: consecutive}, there are at most $3\sqrt{n}/\gamma$ intervals in $\dP_i$ for each $i\in [4]$,
so there are at most $(12\sqrt{n}/\gamma)\cdot (\sqrt{\beta\gamma})^2=12\sqrt{n}\beta$ pairs $\{g_j,g_k\}$ such that $g_j, g_k$ lie in different intervals of some $\dP_i$ and $|j-k|\leq \sqrt{\beta\gamma}$.
We also know that each $\dP_i$ contains at most $2\sqrt{n}\log^2 n$ pairs of type-\1 and type-\2.
Putting all together, as $s\geq (1+o(1)\sqrt{n}$ and $\beta\geq n^{1/4}\log n$, we see that $$|\Phi|\geq (1-o(1))s\beta\sqrt{\log n}-12\sqrt{n}\beta-8\sqrt{n}\log^2 n\geq (1-o(1))s\beta\sqrt{\log n}.$$

Since each $\{g_j,g_k\}$ in $\Phi$ is a normal pair, the difference $g_j\dif g_k$ induces a cycle.
One can see that all such cycles $g_j\dif g_k$ are distinct and thus have different lengths.
Observe that $\sum_{e\in E(G)} \delta(e,g_j,g_k)$ equals the length of the cycle $g_j\dif g_k$, denoted by $|g_j\dif g_k|$.
Therefore, the summation
\begin{equation}\label{equ: 1}
\Sigma=\sum_{\{g_j,g_k\}\in \Phi}\left(\sum_{e\in E(G)} \delta(e,g_j,g_k)\right)=\sum_{\{g_j,g_k\}\in \Phi}|g_j\dif g_k|\geq \sum_{\ell=1}^{|\Phi|}\ell\geq \frac{|\Phi|^2}{2}\geq \left(\frac{1}{2}-o(1)\right) s^2\beta\gamma.
\end{equation}

Next we fix $e\in E(G)$ and estimate the sum of $\delta(e,g_j,g_k)$ over all $\{g_j,g_k\}\in \Phi$.
Note that any paths $g_j,g_k\in \dP_i$ share the same edges out of $E(G(\dP_i))$,
so it suffices for us to consider edges $e\in E(G(\dP_i))$ for some $i\in [4]$ (while for other edges $e$, the above sum always equals zero).

For $e\in E(G(\dP_i))$ with $d(e)\geq \gamma$, by Lemma \ref{lem: consecutive} one can delete $3\beta$ paths in $\dP_i$ such that
there is at most one interval of $\dP_i$ which can contain some remaining paths $g_j, g_k$ with $e\in E(g_j)$ and $e\notin E(g_k)$ and moreover, all such paths $g_j, g_k$ satisfy $j<k$.
Let $\Phi(e)$ be the set of pairs in $\Phi$ containing at least one of the above $3\beta$ paths we delete.
So $|\Phi(e)|\leq 3\beta\cdot 2\sqrt{\beta\gamma}=6\beta\gamma/\sqrt{\log n}$,
and for any $\beta\leq r\leq \sqrt{\beta\gamma}$,
there are at most $r$ pairs $\{g_j,g_k\}$ in $\Phi\backslash \Phi(e)$ satisfying $\delta(e,g_j,g_k)=1$.
For $e\in E(G(\dP_i))$ with $d(e)\leq \beta$, we let $\Phi(e)$ be the set of all pairs in $\Phi$ which contains at least one path containing $e$.
Then it is clear that $|\Phi(e)|\leq \beta\cdot 2\sqrt{\beta\gamma}=2\beta\gamma/\sqrt{\log n}$ and $\delta(e,g_j,g_k)=0$ for all $\{g_j,g_k\}$ in $\Phi\backslash \Phi(e)$.
Hence, for every $e$ with $d(e)\geq \gamma$ or $d(e)\leq \beta$, we have
$$\sum_{\{g_j,g_k\}\in \Phi} \delta(e,g_j,g_k)\le \left(\sum_{\{g_j,g_k\}\in \Phi\bh \Phi(e)} \delta(e,g_j,g_k)\right) +|\Phi(e)|\le \sum_{r=\beta}^{\sqrt{\beta\gamma}}r+6\beta\gamma/\sqrt{\log n}\leq \left(\frac{1}{2}+o(1)\right)\beta\gamma.$$
By Lemma \ref{lem: consecutive}, there are at most $9n\log \log n/\log n$ edges $e$ with $\beta\leq d(e)\leq \gamma$.
Each such $e$ is contained in at most $\gamma$ paths $g_j$, while at most $2\sqrt{\beta\gamma}$ paths $g_k$ can satisfy $\{g_j,g_k\}\in \Phi$.
So at most $\gamma\cdot 2\sqrt{\beta\gamma}=2\beta^2(\log n)^{3/2}$ pairs $\{g_j,g_k\}\in \Phi$ can give $\delta(e,g_j,g_k)=1$.
This implies that
$$\sum_{\beta\leq d(e)\leq \gamma}\sum_{\{g_j,g_k\}\in \Phi} \delta(e,g_j,g_k)\leq \frac{9n\log \log n}{\log n}\cdot 2\beta^2(\log n)^{3/2}=o(n\beta\gamma).$$
Adding all edges $e\in E(G)$ together, we can obtain the following upper bound
\begin{equation}\label{equ: 2}
\Sigma=\sum_{e\in E(G)}\left(\sum_{\{g_j,g_k\}\in \Phi}\delta(e,g_j,g_k)\right)\le (n+s)\cdot\left(\frac{1}{2}+o(1)\right)\beta\gamma + o(n\beta\gamma)= \left(\frac{1}{2}+o(1)\right)n\beta\gamma.
\end{equation}
Combining with \eqref{equ: 1} and \eqref{equ: 2}, we can derive that $s^2\leq (1+o(1))n$ and thus $s\leq (1+o(1))\sqrt{n}$.
This finishes the proof of Theorem \ref{thm: main}.
\qed

\section{Concluding remarks}
In this paper, we prove that any $n$-vertex 2-connected graph $G$ with no two cycles of the same length has at most $n+\sqrt{n}+o(\sqrt{n})$ edges.
We remark that through a more careful calculation, the present proof can show that $G$ contains at most $n+\sqrt{n}+20 \sqrt{n/\log n}$ edges.

We also would like to point out that all statements in Sections~3--5 can hold only under the assumption \eqref{equ:(n-2)cycles}.
Indeed, throughout Sections~3--5, all we need in the proofs is just the upper bound on the total number of cycles,
while the stronger assumption that $G$ contains at most one cycle of length $i$ for each $3\leq i\leq n$ was only used in Section~6.
That also says, the structural constraints we develop in Sections~3--5 also apply to 2-connected graphs $G$ with relatively many edges but few cycles.
In particular, an analog of Lemma \ref{lem: consecutive} still holds for 2-connected graphs $G$ with $n$ vertices, $n+s$ edges and $m$ cycles, assuming that $m\ll s^3$.
This may shed some light on an old problem of Entringer from 1973, which asks to determine all graphs $G$ with exactly one cycle of each length between 3 and $|V(G)|$ (see \cite{BM}, p. 247, Problem 10).

\bibliographystyle{unsrt}

\medskip
	
{\it E-mail address:} jiema@ustc.edu.cn

\medskip
	
{\it E-mail address:} ytc@mail.ustc.edu.cn

\end{document}